\documentclass[12pt,oneside,english,reqno,twoside]{amsart}
\usepackage{lmodern}

\usepackage[T1]{fontenc}
\usepackage[latin9]{luainputenc}
\usepackage{geometry}
\usepackage{amstext}
\usepackage{amsthm}
\usepackage{amssymb}

\makeatletter
\numberwithin{equation}{section}
\numberwithin{figure}{section}
\theoremstyle{plain}
\newtheorem{thm}{\protect\theoremname}[section]
  \theoremstyle{definition}
  \newtheorem{defn}[thm]{\protect\definitionname}
  \theoremstyle{definition}
  \newtheorem{example}[thm]{\protect\examplename}
  \theoremstyle{plain}
  \newtheorem{lem}[thm]{\protect\lemmaname}
  \theoremstyle{plain}
  \newtheorem{cor}[thm]{\protect\corollaryname}
  \theoremstyle{remark}
  \newtheorem*{rem*}{\protect\remarkname}
  \theoremstyle{remark}
  \newtheorem{rem}[thm]{\protect\remarkname}

\usepackage{versions}

\DeclareMathOperator*{\esssup}{ess\,sup}

\DeclareMathOperator*{\esssinf}{ess\,inf}

\DeclareMathOperator*{\esssliminf}{ess\,lim\,inf}

\usepackage{hyperref}

\newcommand{\Addresses}{{
  \bigskip
  \footnotesize
  \textsc{Jarkko Siltakoski,  Department of Mathematics and Statistics, P.O.Box 35, FIN-40014, University of Jyv{\"a}skyl{\"a}, Finland}\par\nopagebreak
  \textit{E-mail address}: \href{mailto:jarkko.j.m.siltakoski@student.jyu.fi}{jarkko.j.m.siltakoski@student.jyu.fi}
}}

\makeatother

\usepackage{babel}
  \providecommand{\corollaryname}{Corollary}
  \providecommand{\definitionname}{Definition}
  \providecommand{\examplename}{Example}
  \providecommand{\lemmaname}{Lemma}
  \providecommand{\remarkname}{Remark}
\providecommand{\theoremname}{Theorem}

\begin{document}
\global\long\def\d{\,d}
\global\long\def\tr{\mathrm{tr}}
\global\long\def\supp{\operatorname{spt}}
\global\long\def\div{\operatorname{div}}
\global\long\def\osc{\operatorname{osc}}
\global\long\def\essup{\esssup}
\global\long\def\aint{\dashint}
\global\long\def\essinf{\esssinf}
\global\long\def\essliminf{\esssliminf}
\global\long\def\sa{|}
\global\long\def\sgn{\operatorname{sgn}}
 \excludeversion{note}\excludeversion{note2}\excludeversion{old}\date{September 2019}

\title[Equivalence between radial solutions]{Equivalence between radial solutions of different non-homogeneous
$p$-Laplacian type equations}

\author{Jarkko Siltakoski}
\begin{abstract}
We study radial viscosity solutions to the equation
\[
-\left|Du\right|^{q-2}\Delta_{p}^{N}u=f(\left|x\right|)\quad\text{in }B_{R}\subset\mathbb{R}^{N},
\]
where $f\in C[0,R)$, $p,q\in(1,\infty)$ and $N\geq2$. Our main
result is that $u(x)=v(\left|x\right|)$ is a bounded viscosity supersolution
if and only if $v$ is a bounded weak supersolution to $-\kappa\Delta_{q}^{d}v=f$
in $(0,R)$, where $\kappa>0$ and $\Delta_{q}^{d}$ is heuristically
speaking the radial $q$-Laplacian in a fictitious dimension $d$.
As a corollary we obtain the uniqueness of radial viscosity solutions.
However, the full uniqueness of solutions remains an open problem.
\end{abstract}

\maketitle

\section{Introduction}

In this paper, we study radial viscosity solutions to the equation
\begin{equation}
-\left|Du\right|^{q-2}\Delta_{p}^{N}u=f(\left|x\right|)\quad\text{in }B_{R},\label{eq:normplap f}
\end{equation}
where
\[
\Delta_{p}^{N}u:=\Delta u+\frac{(p-2)}{\left|Du\right|^{2}}\sum_{i,j=1}^{N}D_{ij}uD_{i}uD_{j}u
\]
is the normalized $p$-Laplacian, $\smash{f\in C[0,R)}$, $\smash{B_{R}\subset\mathbb{R}^{N}}$,
$\smash{N\geq2}$ and $p,q\in(1,\infty)$. The left-hand side of the
equation (\ref{eq:normplap f}) is the usual $p$-Laplacian when $q=p$
and the normalized $p$-Laplacian when $q=2$. In particular, the
equation (\ref{eq:normplap f}) may be in a non-divergence form and
therefore the use of viscosity solutions is appropriate. Since we
are interested in radial solutions, it is natural to restrict to a
ball at the origin and assume that the source term is radial.

Recently Parviainen and V{\'a}zquez \cite{ParvVaz} proved that radial
viscosity solutions to the parabolic equation $\partial_{t}u=\left|Du\right|^{q-2}\Delta_{p}^{N}u$
coincide with weak solutions of a one dimensional equation related
to the usual radial $q$-Laplacian. The objective of the present work
is to obtain a similar equivalence result for the equation (\ref{eq:normplap f})
while also considering supersolutions. Since the one dimensional equation
satisfies a comparison principle, we obtain the uniqueness of radial
solutions to (\ref{eq:normplap f}) as a corollary. To the best of
our knowledge, this was previously known only for $f=0$ or $f$ with
a constant sign \cite{kawohlManfrediParviainen12} and the full uniqueness
remains an open problem.

Stated more precisely, our main result is that $u(x):=v(\left|x\right|)$
is a bounded viscosity supersolution to (\ref{eq:normplap f}) if
and only if $v$ is a bounded weak supersolution to the one-dimensional
equation
\begin{equation}
-\kappa\Delta_{q}^{d}v\geq f\quad\text{in }(0,R)\subset\mathbb{R},\label{eq:radial eq}
\end{equation}
where
\[
\Delta_{q}^{d}v:=\left|v^{\prime}\right|^{q-2}\big((q-1)v^{\prime\prime}+\frac{d-1}{r}v^{\prime}\big)
\]
and $\kappa$ and $d$ are given in (\ref{eq:kappa and d}). Heuristically
speaking, the operator $\smash{\Delta_{q}^{d}}$ is the usual radial
$q$-Laplacian in a fictitious dimension $d$. Indeed, we show that
if $d$ is an integer, then supersolutions to (\ref{eq:radial eq})
coincide with radial supersolutions to the equation $\smash{-\Delta_{q}u\geq f(\left|x\right|)}$
in $\smash{B_{R}\subset\mathbb{R}^{d}}$. The precise definition of
weak supersolutions to (\ref{eq:radial eq}) uses certain weighted
Sobolev spaces and is given in Section 2.

Let us illustrate the relationship between equations (\ref{eq:normplap f})
and (\ref{eq:radial eq}) by a few formal computations. Assume that
$\smash{u:\mathbb{R}^{N}\rightarrow\mathbb{R}}$ is a smooth function
such that $u(x)=v(\left|x\right|)$ for some $v:[0,\infty)\rightarrow\mathbb{\mathbb{R}}$.
Then by a simple calculation, we have $\smash{Du(re_{1})=e_{1}v^{\prime}(r)}$
and $\smash{D^{2}u(re_{1})=e_{1}\otimes e_{1}v^{\prime\prime}(r)+r^{-1}(I-e_{1}\otimes e_{1})v^{\prime}(r)}$
for $r>0$. In particular we have $\smash{\left|Du(re_{1})\right|=\left|v^{\prime}(r)\right|}$.
Assuming that the gradient does not vanish, we obtain
\begin{align}
\Delta_{p}^{N}u(re_{1}) & =\Delta u+\frac{p-2}{\left|Du(re_{1})\right|^{2}}\sum_{i,j=1}^{N}D_{ij}uD_{i}uD_{j}u\nonumber \\
 & =(p-1)v^{\prime\prime}(r)+\frac{N-1}{r}v^{\prime}(r).\label{eq:intro 1}
\end{align}
Denoting
\begin{equation}
\kappa:=\frac{p-1}{q-1},\quad d:=\frac{(N-1)(q-1)}{p-1}+1\label{eq:kappa and d}
\end{equation}
and multiplying (\ref{eq:intro 1}) by $\left|Du(re_{1})\right|^{q-2}$,
it follows that
\begin{align*}
\left|Du(re_{1})\right|^{q-2}\Delta_{p}^{N}u(re_{1}) & =\kappa\left|v^{\prime}(r)\right|^{q-2}\big((q-1)v^{\prime\prime}(r)+\frac{d-1}{r}v^{\prime}(r)\big),
\end{align*}
where the right-hand side equals $\smash{\kappa\Delta_{q}^{d}v(r)}$.
Thus at least formally there is an equivalence between the equations
(\ref{eq:normplap f}) and (\ref{eq:radial eq}). However, to make
this rigorous, we need to carefully exploit the exact definitions
of viscosity and weak supersolutions.

To show that $v$ is a weak supersolution to (\ref{eq:radial eq})
whenever $u$ is a viscosity supersolution to (\ref{eq:normplap f}),
we apply the method developed by Julin and Juutinen \cite{newequivalence}.
The idea is to approximate $u$ using its inf-convolution $\smash{u_{\varepsilon}}$.
Since $\smash{u_{\varepsilon}}$ is still radial, there is $\smash{v_{\varepsilon}}$
such that $\smash{u_{\varepsilon}(x)=v_{\varepsilon}(\left|x\right|)}$.
Using the pointwise properties of inf-convolution, we show that $\smash{v_{\varepsilon}}$
is a weak supersolution to (\ref{eq:radial eq}). It then suffices
to pass to the limit to see that $v$ is also a weak supersolution.

To show that $u$ is a viscosity supersolution to (\ref{eq:normplap f})
whenever $v$ is a weak supersolution to (\ref{eq:radial eq}), we
adapt a standard argument used for example in \cite{equivalence_plaplace}.
Thriving for a contradiction, we assume that $u$ is not a viscosity
supersolution. Roughly speaking, this means that there exists a smooth
function $\varphi$ that touches $u$ from below and (\ref{eq:normplap f})
fails at the point of touching. We use $\varphi$ to construct a new
function $\phi$ that is a weak subsolution to (\ref{eq:radial eq})
and touches $v$ from below. This violates a comparison principle
and produces the desired contradiction. To avoid technicalities that
might occur should the gradient of $\varphi$ vanish, we use an equivalent
definition of viscosity supersolutions proposed by Birindelli and
Demengel \cite{birindelliDemengel04}. Extra care is also needed if
the point of touching is the origin.

The equation (\ref{eq:normplap f}) has received an increasing amount
of attention in the last several years. For example, the $\smash{C^{1,\alpha}}$
regularity of radial solutions to (\ref{eq:normplap f}) was shown
by Birindelli and Demengel \cite{birindelliDemengel12}. Using a different
technique Imbert and Silvestre \cite{imbertSilvestre12} proved the
$\smash{C^{1,\alpha}}$ regularity of solutions to $\smash{\left|Du\right|^{q-2}F(D^{2}u)=f}$
when $q>2$. More recently Attouchi and Ruosteenoja \cite{attouchiRuosteenoja18}
obtained $\smash{C^{1,\alpha}}$ regularity results for any solution
of (\ref{eq:normplap f}) and also proved some $\smash{W^{2,2}}$
estimates.

The equivalence of viscosity and weak solutions was first studied
by Ishii \cite{Ishii95} in the case of linear equations. The equivalence
of solutions for $p$-Laplace equation was first obtained by Manfredi,
Lindqvist and Juutinen \cite{equivalence_plaplace}, later in a different
way by Julin and Juutinen \cite{newequivalence} and for the $p(x)$-Laplace
equation by Juutinen, Lukkari and Parviainen \cite{equivalence_p(x)Laplace}.
Recent papers on this matter include the works of Attouchi, Parviainen
and Ruosteenoja \cite{OptimalC1} on the normalized $p$-Poisson problem
where the equivalence was used to obtain $C^{1,\alpha}$ regularity
of solutions, Medina and Ochoa \cite{chilepaper} on a non-homogeneous
$p$-Laplace equation, Siltakoski \cite{siltakoski18} on a normalized
$p(x)$-Laplace equation and Bieske and Freeman \cite{bieskeFreeman18}
on the $p(x)$-Laplace equation in Carnot groups.

The paper is organized as follows. Section 2 contains the precise
definitions of viscosity solutions and weak solutions in our context.
In Section 3 we show that weak supersolutions to (\ref{eq:radial eq})
are viscosity supersolutions to (\ref{eq:normplap f}) and the converse
is proved in Section 4. In Section 5 we consider the special case
where $d$ is an integer and finally the Appendix contains some properties
of the weighted Sobolev spaces. 

\section{Preliminaries}

\subsection{Viscosity solutions}

Let $\smash{\varphi,u:B_{R}\rightarrow\mathbb{R}}$. We say that \textit{$\varphi$
touches $u$ from below at $\smash{x_{0}\in B_{R}}$} if $\smash{\varphi(x_{0})=u(x_{0})}$
and $\varphi(x)<u(x)$ when $\smash{x\not=x_{0}}$.
\begin{defn}
\label{def:visc sol}A bounded lower semicontinuous function $\smash{u:B_{R}\rightarrow\mathbb{R}}$
is a \textit{viscosity supersolution} to $\eqref{eq:normplap f}$
in $\smash{B_{R}}$ if whenever $\smash{\varphi\in C^{2}}$ touches
$u$ from below at $\smash{x_{0}}$ and $D\varphi(x)\not=0$ when
$\smash{x\not=x_{0}}$, then we have
\[
\limsup_{x_{0}\not=y\rightarrow x_{0}}\left(-\left|D\varphi(y)\right|^{q-2}\Delta_{p}^{N}\varphi(y)\right)-f(\sa x_{0}\sa)\geq0.
\]
A bounded upper semicontinuous function $\smash{u:B_{R}\rightarrow\mathbb{R}}$
is \textit{a viscosity subsolution} to $\eqref{eq:normplap f}$ in
$\smash{B_{R}}$ if whenever $\smash{\varphi\in C^{2}}$ touches $u$
from above at $\smash{x_{0}}$ and $D\varphi(x)\not=0$ when $\smash{x\not=x_{0}}$,
then we have
\[
\liminf_{x_{0}\not=y\rightarrow x_{0}}\left(-\left|D\varphi(y)\right|^{q-2}\Delta_{p}^{N}\varphi(y)\right)-f(\sa x_{0}\sa)\leq0.
\]
A function is a \textit{viscosity solution} if it is both viscosity
sub- and supersolution.
\end{defn}
The limit procedure in Definition \ref{def:visc sol} is needed because
of the discontinuity in the equation when $q\leq2$. When $q>2$ the
equation is continuous and the limit procedure is unnecessary.

\subsection{Weak solutions}

In order to define weak solutions, we must first define the appropriate
Sobolev spaces. The \textit{weighted Lebesgue space} $\smash{L^{q}(r^{d-1},(0,R))}$
is defined as the set of all measurable functions $\smash{v:(0,R)\rightarrow\mathbb{R}}$
such that the norm
\[
\left\Vert v\right\Vert _{L^{q}(r^{d-1},(0,R))}:=\big(\int_{0}^{R}\left|v\right|^{q}r^{d-1}\d r\big)^{1/q}
\]
is finite. We define the \textit{weighted Sobolev space} $\smash{W^{1,q}(r^{d-1},(0,R))}$
as the set of all functions $\smash{v\in L^{q}(r^{d-1},(0,R))}$ whose
distributional derivative $\smash{v^{\prime}}$ is in $\smash{L^{q}(r^{d-1},(0,R))}$.
As usual, by distributional derivative we mean that $v^{\prime}$
satisfies
\[
\int_{0}^{R}v^{\prime}\varphi\d r=-\int_{0}^{R}v\varphi^{\prime}\d r
\]
 for all $\varphi\in C_{0}^{\infty}(0,R)$. We equip $\smash{W^{1,q}(r^{d-1},(0,R))}$
with the norm
\[
\left\Vert v\right\Vert _{W^{1,q}(r^{d-1},(0,R))}:=\big(\int_{0}^{R}\left|v\right|^{q}r^{d-1}\d r+\int_{0}^{R}\sa v^{\prime}\sa^{q}r^{d-1}\d r\big)^{1/q}.
\]
Then $\smash{W^{1,q}(r^{d-1},(0,R))}$ is a separable Banach space,
see e.g.\ \cite{kufnerOpic84} or \cite{kufner85}. Since $d>1$,
it follows from Theorem 7.4 in \cite{kufner85} that the set
\[
C^{\infty}[0,R]:=\left\{ u_{|(0,R)}:u\in C^{\infty}(\mathbb{R})\right\} 
\]
is dense in $\smash{W^{1,q}(r^{d-1},(0,R))}$. For the benefit of
the reader we have also included a proof in the appendix, see Theorem
\ref{thm:w1q density}. We point out that any $v\in W^{1,q}(r^{d-1},(0,R))$
has a representative that is continuous in $(0,R]$. Indeed, for any
$\delta>0$ we have $\smash{\delta^{d-1}<r^{d-1}<R^{d-1}}$ when $\smash{r\in(\delta,R)}$
and consequently the restriction $\smash{v_{|(\delta,R)}}$ is in
the usual Sobolev space $\smash{W^{1,q}(\delta,R)}$. 

In addition to \cite{kufner85}, weighted Sobolev spaces have been
studied for example in \cite{superbook}. However, the weight $w:\mathbb{R}\rightarrow\mathbb{R}$,
$\smash{w(x)=\left|x\right|^{d-1}}$ is not necessarily $q$-admissible
in the sense of \cite{superbook}. Indeed, in the one dimensional
setting $q$-admissible weights coincide with Muckenhoupt's $\smash{A_{q}}$-weights
\cite{bjornBuckleyKeith06}. Thus $w$ is $q$-admissible if and only
if $d-1<p-1$ which by (\ref{eq:kappa and d}) is equivalent to $p>N$.

With the weighted Sobolev spaces at hand, we can define weak solutions.
Recall that formally the equation (\ref{eq:radial eq}) reads as 
\[
-\kappa\left|v^{\prime}\right|^{q-2}\big((q-1)v^{\prime\prime}+\frac{d-1}{r}v^{\prime}\big)=f\quad\text{in }(0,R),
\]
where $\kappa$ and $d$ are the constants given in (\ref{eq:kappa and d}).
If $v$ is smooth and the gradient does not vanish, this can be equivalently
written as 
\[
-\kappa\big(\left|v^{\prime}\right|^{q-2}v^{\prime}r^{d-1}\big)^{\prime}-fr^{d-1}=0\quad\text{in }(0,R).
\]

\begin{defn}
\label{def:weak sol}We say that $v$ is a \textit{weak supersolution}
to (\ref{eq:radial eq}) in $(0,R)$ if $v\in\smash{W^{1,q}(r^{d-1},(0,R^{\prime}))}$
for all $\smash{R^{\prime}\in(0,R)}$ and we have
\begin{equation}
\int_{0}^{R}\kappa\left|v^{\prime}\right|^{q-2}v^{\prime}\varphi^{\prime}r^{d-1}-\varphi fr^{d-1}\d r\geq0\label{eq:weak super}
\end{equation}
for all non-negative $\smash{\varphi\in C_{0}^{\infty}(-R,R)}$. For
\textit{weak subsolutions} the inequality (\ref{eq:weak super}) is
reversed. Furthermore, $\smash{v\in C[0,R)}$ is a \textit{weak solution}
if it is both weak sub- and supersolution.
\end{defn}
Recall that our goal is to establish an equivalence between radial
viscosity supersolutions of (\ref{eq:normplap f}) and weak supersolutions
of (\ref{eq:radial eq}). For this reason the class of test functions
in Definition \ref{def:weak sol} needs to be $\smash{C_{0}^{\infty}(-R,R)}$
instead of $\smash{C_{0}^{\infty}(0,R)}$, see the example below.
We also point out that if $d$ is an integer, then weak supersolutions
in the sense of Definition 2.2 coincide with radial weak supersolutions
to $\smash{\Delta_{q}u\geq f(\left|x\right|)}$, where $\smash{\Delta_{q}}$
is the usual $q$-Laplacian in $d$-dimensions, see Theorem \ref{thm:q laplace equivalence}. 
\begin{example}
Let $p>N$, $f\equiv0$, and define $v:(0,R)\rightarrow\mathbb{R}$
by
\[
v(r):=\frac{1}{1-\alpha}r^{1-\alpha},\quad\text{where }\alpha:=\frac{N-1}{p-1}.
\]
Then $\smash{v\in W^{1,q}(r^{d-1},(0,R))}$ and it satisfies (\ref{eq:weak super})
for all non-negative $\smash{\varphi\in C_{0}^{\infty}(0,R)}$, but
$u(x):=v(\left|x\right|)$ is not a viscosity supersolution to (\ref{eq:normplap f}).
To verify this, observe first that $v$ is in the correct Sobolev
space. Indeed, the distributional derivative of $v$ is $\smash{v^{\prime}(r)=r^{-\alpha}}$
and thus $v^{\prime}\in L^{q}(r^{d-1},(0,R))$ since
\[
-\alpha q+d-1=-\frac{N-1}{p-1}q+\frac{(N-1)(q-1)}{p-1}=-\frac{N-1}{p-1}>-1.
\]
Moreover, for any non-negative $\varphi\in C_{0}^{\infty}(0,R)$,
we have
\[
\int_{0}^{R}\kappa\left|v^{\prime}\right|^{q-2}v^{\prime}\varphi^{\prime}r^{d-1}\d r=\int_{0}^{R}\kappa r^{-(q-1)\alpha}\varphi^{\prime}r^{\frac{(N-1)(q-1)}{p-1}}\d r=\int_{0}^{R}\kappa\varphi^{\prime}\d r=0.
\]
To see that the function $u(x)=v(\left|x\right|)$ is not a viscosity
supersolution to (\ref{eq:normplap f}), set $\smash{\phi(x):=(x_{1}-1)^{2}}$.
Then $u-\phi$ has a local minimum at $0$ and $D\phi(0)\not=0$,
but
\begin{align*}
 & -\left|D\varphi(0)\right|^{q-2}\Delta_{p}^{N}\varphi(0)\\
 & \ =-\left|2\right|^{q-2}\big(\tr(2e_{1}\otimes e_{1})+\frac{(p-2)}{2^{2}}(-2e_{1})^{\prime}(2e_{1}\otimes e_{1})(-2e_{1})\big)<0,
\end{align*}
which means that $u$ is not a supersolution.
\end{example}
\begin{lem}
\label{lem:admissible test functions}We may extend the class of test
functions in Definition \ref{def:weak sol} to $\varphi\in\smash{W^{1,q}(r^{d-1},(0,R))}$
such that $\smash{\supp\varphi\subset[0,R^{\prime})}$ for some $\smash{R^{\prime}\in(0,R)}$.
\end{lem}
\begin{proof}
Take a cut-off function $\xi\in C_{0}^{\infty}(-R,R)$ such that $\xi\equiv1$
in $[0,R^{\prime}]$. Take $\varphi_{j}\in C^{\infty}[0,R]$ such
that $\varphi_{j}\rightarrow\varphi$ in $W^{1,q}(r^{d-1},(0,R))$.
Set $\phi_{j}:=\varphi_{j}\xi$. Then $\phi_{j}\in C_{0}^{\infty}(-R,R)$
and hence 
\begin{align}
0\leq & \int_{0}^{R}\left|v^{\prime}\right|^{q-2}v^{\prime}\phi_{j}^{\prime}r^{d-1}-f\phi_{j}r^{d-1}\d r\nonumber \\
= & \int_{0}^{R}\left|v^{\prime}\right|^{q-2}v^{\prime}\varphi_{j}^{\prime}\xi r^{d-1}-f\varphi_{j}\xi r^{d-1}\d r+\int_{0}^{R}\left|v^{\prime}\right|^{q-2}v^{\prime}\varphi_{j}\xi^{\prime}r^{d-1}\d r.\label{eq:admissible test functions 1}
\end{align}
Since $\xi\equiv1$ in $\supp\varphi$, we have $\varphi^{\prime}\xi=\varphi^{\prime}$
and so\begin{note2}(We have $\varphi^{\prime}\xi=\varphi^{\prime}$
in $\supp\varphi$ and $\varphi^{\prime}=0$ a.e.\ in $(0,R)\setminus\supp\varphi$.
Indeed, let $\psi=\varphi_{\text{|(\ensuremath{\delta},R)}}\in W^{1,q}(\delta,R)$.
Then $\psi\in W^{1,q}(\delta,R)$ and $\psi^{\prime}=0$ a.e.\ in
$(\delta,R)\setminus\supp\psi$. Thus $\varphi^{\prime}\xi=\varphi^{\prime}$)\end{note2}

\begin{align*}
 & \int_{0}^{R}\left|v^{\prime}\right|^{q-2}v^{\prime}\varphi_{j}^{\prime}\xi r^{d-1}-f\varphi_{j}\xi r^{d-1}\d r\\
 & \ =\int_{0}^{R}\left|v^{\prime}\right|^{q-2}v^{\prime}(\varphi_{j}^{\prime}-\varphi^{\prime})\xi r^{d-1}-f(\varphi_{j}-\varphi)\xi r^{d-1}\d r\\
 & \ \ \ \ \ +\int_{0}^{R}\left|v^{\prime}\right|^{q-2}v^{\prime}\varphi^{\prime}r^{d-1}-f\varphi r^{d-1}\d r.
\end{align*}
Combining this with (\ref{eq:admissible test functions 1}), we get
\begin{align*}
\int_{0}^{R}\left|v^{\prime}\right|^{q-2}v^{\prime}\varphi^{\prime}r^{d-1}-f\varphi r^{d-1}\d r\geq- & \!\int_{0}^{R}\left|v^{\prime}\right|^{q-1}\left|\varphi_{j}^{\prime}-\varphi^{\prime}\right|\xi r^{d-1}+\left|f\right|\left|\varphi_{j}-\varphi\right|\xi r^{d-1}dr\\
 & -\!\int_{0}^{R}\left|v^{\prime}\right|^{q-1}\left|\varphi_{j}\right|\left|\xi^{\prime}\right|r^{d-1}\d r.
\end{align*}
The first integral at the right-hand side converges to zero by H{\"o}lder's
inequality. Moreover, since $\smash{\varphi\xi^{\prime}\equiv0}$
in $(0,R)$, we have
\[
\int_{0}^{R}\left|v^{\prime}\right|^{q-1}\left|\varphi_{j}\right|\left|\xi^{\prime}\right|r^{d-1}\d r=\int_{0}^{R}\left|v^{\prime}\right|^{q-1}\left|\varphi_{j}-\varphi\right|\left|\xi^{\prime}\right|r^{d-1}\d r\rightarrow0.\qedhere
\]
\end{proof}

\section{Weak solutions are viscosity solutions}

We show that bounded weak supersolutions to (\ref{eq:radial eq})
are radial viscosity supersolutions to (\ref{eq:normplap f}). In
order to formulate the precise statement, we recall that the lower
semicontinuous reguralization of a function $v:(0,R)\rightarrow\mathbb{R}$
is defined by
\[
v_{\ast}(r):=\essliminf_{s\rightarrow r}v(s):=\lim_{S\rightarrow0}\essinf_{s\in(r-S,r+S)\cap(0,R)}v(s)
\]
for all $r\in[0,R]$. Observe that since any function $\smash{v\in W^{1,q}(r^{d-1},(0,R))}$
admits a continuous representative, we have $\smash{v=v_{\ast}}$
almost everywhere in $(0,R)$ for such $v$. 
\begin{thm}
\label{thm:weak is visc}Assume that $v$ is a bounded weak supersolution
to (\ref{eq:radial eq}) in $(0,R)$. Then $u(x):=v_{\ast}(\left|x\right|)$
is a viscosity supersolution to (\ref{eq:normplap f}) in $B_{R}$.
\end{thm}
To prove Theorem \ref{thm:weak is visc}, we use the following definition
of viscosity supersolutions introduced by Birindelli and Demengel
\cite{birindelliDemengel04}. Its advantage is that we may restrict
to test functions whose gradient does not vanish. It is shown in \cite{attouchiRuosteenoja18}
that Definitions \ref{def:visc sol} and \ref{def:visc sol 2} are
equivalent.
\begin{defn}
\label{def:visc sol 2}A bounded and lower semicontinuous function
$u:B_{R}\rightarrow\mathbb{R}$ is a viscosity supersolution to (\ref{eq:normplap f})
if for any $x_{0}\in B_{R}$ one of the following conditions holds.
\begin{enumerate}
\item The function $u$ is not a constant in $B_{\delta}(x_{0})$ for any
$\delta>0$, and whenever $\varphi\in C^{2}$ touches $u$ from below
at $x_{0}$ with $D\varphi(x_{0})\not=0$, we have
\begin{equation}
-\left|D\varphi(x_{0})\right|^{q-2}\Delta_{p}^{N}\varphi(x_{0})\geq f(\left|x_{0}\right|).\label{eq:somecond}
\end{equation}
\item The function $u$ is a constant in $B_{\delta}(x_{0})$ for some $\delta>0$,
and we have
\[
f(\left|x\right|)\leq0\quad\text{for all }x\in B_{\delta}(x_{0}).
\]
\end{enumerate}
\end{defn}
\begin{proof}[Proof of Theorem \ref{thm:weak is visc}]
Let $x_{0}\in B_{R}$. We first consider the case where $u$ is a
constant in $B_{\delta}(x_{0})$ for some $\delta>0$. In this case
also $v$ is a constant a.e.\ in $I:=(0,R)\cap(\left|x_{0}\right|-\delta,\left|x_{0}\right|+\delta)$.
This implies that $\smash{v^{\prime}\equiv0}$ a.e.\ in $I$ and
thus, since $v$ is a weak supersolution to (\ref{eq:radial eq}),
we have
\[
\int_{I}\varphi fr^{d-1}\d r\leq0
\]
for all non-negative $\varphi\in C_{0}^{\infty}(I)$. Since $f$ is
continuous, it follows that $f\le0$ in $I$ and consequently $f(\left|x\right|)\leq0$
in $B_{\delta}(x_{0})$, as desired.

Assume then that $u$ is not a constant near $\smash{x_{0}}$. Suppose
on the contrary that the condition (i) of Definition \ref{def:visc sol 2}
fails at $\smash{x_{0}}$, that is, there exists $\smash{\varphi\in C^{2}}$
touching $u$ from below at $x_{0}$ with $\smash{D\varphi(x_{0})\not=0}$
and
\begin{equation}
f(\left|x_{0}\right|)>-\left|D\varphi(x_{0})\right|^{q-2}\Delta_{p}^{N}\varphi(x_{0}).\label{eq:weak is visc 2}
\end{equation}
We consider the case $\smash{x_{0}\not=0}$ first and argue like in
the proof of Proposition A.3 in \cite{ParvVaz}. Let $Q$ be an orthogonal
matrix such that $\smash{x_{0}}=\smash{r_{0}Qe_{1}}$, where $\smash{r_{0}:=\left|x_{0}\right|}$
and define $\smash{\psi(x):=\varphi(Qx)}$. Then $\psi$ touches $u$
from below at $\smash{r_{0}e_{1}}$ and we have $\smash{D\psi(x)=Q^{\prime}D\varphi(Qx)}$
and $\smash{D^{2}\psi(x)=Q^{\prime}D^{2}\varphi(Qx)Q}$. From these
and (\ref{eq:weak is visc 2}) it follows that 
\begin{equation}
f(r_{0})>-\left|D\psi(r_{0}e_{1})\right|^{q-2}\Delta_{p}^{N}\psi(r_{0}e_{1}).\label{eq:weak is visc some}
\end{equation}
Since $\psi$ touches the radial function $u$ from below at $r_{0}e_{1}\not=0$,
we have $\smash{D_{i}\psi(r_{0}e_{1})=0}$ and $\smash{D_{ii}\psi(r_{0}e_{1})\leq\frac{1}{r_{0}}D_{1}\psi(r_{0}e_{1})}$
for $1<i\leq N$ (see Lemma \ref{lem:jets of radial function} below).
Thus by setting $\smash{\phi(r):=\psi(re_{1})}$, we obtain from (\ref{eq:weak is visc some})
\begin{align*}
f(r_{0}) & >-\left|D_{1}\psi(r_{0}e_{1})\right|^{q-2}\big(D_{11}\psi(r_{0}e_{1})+\sum_{i=2}^{N}D_{ii}\psi(r_{0}e_{1})+(p-2)D_{11}\psi(r_{0}e_{1})\big)\\
 & \geq-\left|D_{1}\psi(r_{0}e_{1})\right|^{q-2}\big((p-1)D_{11}\psi(r_{0}e_{1})+\frac{N-1}{r_{0}}D_{1}\psi(r_{0}e_{1})\big)\\
 & =-\kappa\left|\phi^{\prime}(r_{0})\right|^{q-2}\big((q-1)\phi^{\prime\prime}(r_{0})+\frac{d-1}{r_{0}}\phi^{\prime}(r_{0})\big),
\end{align*}
where we used that $\smash{\kappa=\frac{p-1}{q-1}}$ and $\smash{d=\frac{(N-1)(q-1)}{p-1}+1}$.
Since the above inequality is strict, by continuity it remains true
in some interval $I\Subset(0,R)$ containing $r_{0}$. In other words,
for any $r\in I$ it holds that
\begin{align*}
f(r)r^{d-1}> & -\kappa\left|\phi^{\prime}(r)\right|^{q-2}\big((q-1)\phi^{\prime\prime}(r)+\frac{d-1}{r}\phi^{\prime}(r)\big)r^{d-1}\\
= & -\kappa\big(\left|\phi^{\prime}(r)\right|^{q-2}\phi^{\prime}(r)r^{d-1}\big)^{\prime}.
\end{align*}
Multiplying this by a non-negative function $\eta\in C_{0}^{\infty}(I)$
and integrating by parts we find that
\begin{equation}
\int_{I}\kappa\left|\phi^{\prime}\right|^{q-2}\phi^{\prime}\eta^{\prime}r^{d-1}-\eta fr^{d-1}\d r\leq0.\label{eq:weak is visc 3}
\end{equation}
We set
\[
\overline{\phi}(r):=\phi(r)+l,
\]
where $\smash{l:=\min_{r\in\partial I}(v_{\ast}(r)-\phi(r))>0}$.
Then $\overline{\phi}$ still satisfies (\ref{eq:weak is visc 3}).
Since $\smash{\overline{\phi}\leq v_{\ast}}$ on $\partial I$, it
follows from a comparison principle that $\smash{\overline{\phi}\leq v_{\ast}}$
in $I$ (see Lemma \ref{lem:comparison lemma} below). But this is
a contradiction since $\smash{\overline{\phi}(r_{0})=v_{\ast}(r_{0})}$
and $l>0$.

Consider then the case $x_{0}=0$. Denote $\xi:=D\varphi(0)/\left|D\varphi(0)\right|$
and define a function $\phi:[0,R)\rightarrow\mathbb{R}$ by
\[
\phi(r):=\varphi(r\xi).
\]
Then for $r>0$ we have
\[
\phi^{\prime}(r)=\xi\cdot D\varphi(r\xi)\quad\text{and}\quad\phi^{\prime\prime}(r)=\xi^{\prime}D^{2}\varphi(r\xi)\xi.
\]
Since $\xi\cdot D\varphi(0)=\left|D\varphi(0)\right|>0$, it follows
by continuity that there are constants $M,\delta>0$ such that
\begin{equation}
\phi^{\prime}(r)\geq M\quad\text{when }r\in(0,\delta).\label{eq:weak is visc 5}
\end{equation}
Hence the quantity $\smash{(d-1)r^{-1}\phi^{\prime}(r)}$ is large
when $r>0$ is small. Therefore, since $f$ and $\smash{\phi^{\prime\prime}}$
are bounded in $\smash{(0,R^{\prime})}$ for any $\smash{R^{\prime}<R}$,
there exists $\delta>0$ such that for all $r\in(0,\delta)$ we have
\begin{align*}
f(r)\geq & -\kappa\left|\phi^{\prime}(r)\right|^{q-2}\big((q-1)\phi^{\prime\prime}(r)+\frac{d-1}{r}\phi^{\prime}(r)\big)\\
= & -\kappa\big(\left|\phi^{\prime}(r)\right|^{q-2}\phi^{\prime}(r)r^{d-1}\big)^{\prime}r^{1-d}.
\end{align*}
In other words, for all $r\in(0,\delta)$ it holds that
\begin{equation}
-\kappa\big(\left|\phi^{\prime}(r)\right|^{q-2}\phi^{\prime}(r)r^{d-1}\big)^{\prime}-f(r)r^{d-1}\leq0.\label{eq:weak is visc 4}
\end{equation}
On the other hand, since $\phi^{\prime}$ is bounded in $(0,\delta)$
we have $\phi\in W^{1,q}(r^{d-1},(0,\delta))$. Moreover, for any
non-negative $\smash{\zeta\in C_{0}^{\infty}(-\delta,\delta)}$ we
obtain using integration by parts
\begin{align*}
 & \int_{0}^{\delta}\kappa\left|\phi^{\prime}\right|^{q-2}\phi^{\prime}\zeta^{\prime}r^{d-1}-\zeta fr^{d-1}\d r\\
 & \ =\lim_{h\rightarrow0}\int_{h}^{\delta}\kappa\left|\phi^{\prime}\right|^{q-2}\phi^{\prime}\zeta^{\prime}r^{d-1}-\zeta fr^{d-1}\d r\\
 & \ =\lim_{h\rightarrow0}\bigg(\int_{h}^{\delta}-\kappa\big(\left|\phi^{\prime}\right|^{q-2}\phi^{\prime}r^{d-1}\big)^{\prime}\zeta-\zeta fr^{d-1}\d r-\kappa\left|\phi^{\prime}(h)\right|^{q-2}\phi^{\prime}(h)h^{d-1}\zeta(h)\bigg)\leq0,
\end{align*}
where we used (\ref{eq:weak is visc 4}) and noticed that the last
term converges to zero because $d-1>0$ and $\smash{\phi^{\prime}\geq M>0}$
in $(0,\delta)$. Thus $\phi$ is a weak subsolution to (\ref{eq:radial eq})
in $(0,\delta)$. We set
\[
\overline{\phi}(r):=\phi(r)+l,
\]
where $\smash{l:=\phi(\delta)-v_{\ast}(\delta)>0.}$ Then $\smash{\overline{\phi}\leq v_{\ast}}$
in $(0,\delta)$ by Theorem \ref{thm:comparison principle}. Hence
it follows from continuity of $\smash{\overline{\phi}}$ and definition
of $\smash{v_{\ast}}$ that $\smash{\overline{\phi}(0)\leq v_{\ast}(0)}$.
But this is a contradiction since $\smash{\overline{\phi}(0)=\varphi(0)+l=v_{\ast}(0)+l}$
and $l>0$.
\end{proof}
We still need to prove the lemmas used in the previous proof: the
comparison theorems and the following fact about the derivatives of
test functions.
\begin{lem}
\label{lem:jets of radial function}Let $u:B_{R}\rightarrow\mathbb{R}$
be radial. Assume that $\varphi\in C^{2}$ touches $u$ from below
at $re_{1}\not=0$. Then for $1<i\leq N$ we have
\[
D_{i}\varphi(re_{1})=0\quad\text{and}\quad D_{ii}\varphi(re_{1})\leq\frac{1}{r}D_{1}\varphi(re_{1}).
\]
\end{lem}
\begin{proof}
Since $\varphi\in C^{2}$, we have
\[
\varphi(y)=\varphi(re_{1})+(y-re_{1})\cdot D\varphi(re_{1})+\frac{1}{2}(y-re_{1})^{\prime}D^{2}\varphi(re_{1})(y-re_{1})+o(\left|y-re_{1}\right|^{2})
\]
as $y\rightarrow re_{1}$. Letting $y=re_{1}+he_{i}$, where $h>0$
and $1<i\leq N$, the above implies that 
\begin{equation}
hD_{i}\varphi(re_{1})+\frac{1}{2}h^{2}D_{ii}\varphi(re_{1})=\varphi(re_{1}+he_{i})-\varphi(re_{1})+o(\left|h\right|^{2})\text{ as }h\rightarrow0.\label{eq:jets of radial function 1}
\end{equation}
Let now 
\[
S(h):=r-\sqrt{r^{2}-h^{2}}
\]
so that the vector $re_{1}+he_{i}-S(h)e_{1}$ lies on the boundary
of the ball $B_{r}(0)$. Since $u$ is constant on $\partial B_{r}(0)$,
the assumption that $\varphi$ touches $u$ from below at $re_{1}$
implies
\[
\varphi(re_{1})=u(re_{1})=u(re_{1}+he_{i}-S(h)e_{1})\geq\varphi(re_{1}+he_{i}-S(h)e_{1}).
\]
Combining this with (\ref{eq:jets of radial function 1}) we obtain
\begin{equation}
hD_{i}\varphi(re_{1})+\frac{1}{2}h^{2}D_{ii}\varphi(re_{1})\leq\varphi(re_{1}+he_{i})-\varphi(re_{1}+he_{i}-S(h)e_{1})+o(\left|h\right|^{2}).\label{eq:jets of radial function 2}
\end{equation}
Since $\varphi\in C^{2}$, there is $M>0$ such that for all $a,z\in B_{1}(re_{1})$
we have the estimate
\[
\varphi(a)-\varphi(z)\leq-(z-a)\cdot D\varphi(a)+M\left|z-a\right|^{2}.
\]
Setting $a=re_{1}+he_{i}$ and $z=re_{1}+he_{i}-S(h)e_{1}$, the above
and (\ref{eq:jets of radial function 2}) lead to
\[
\frac{1}{h}D_{i}\varphi(re_{1})+\frac{1}{2}D_{ii}\varphi(re_{1})\leq\frac{S(h)}{h^{2}}D_{1}\varphi(re_{1}+he_{i})+M\frac{\left|S(h)\right|}{h^{2}}^{2}+\frac{o(\left|h\right|^{2})}{h^{2}}.
\]
Observe that $\frac{S(h)}{h^{2}}\rightarrow\frac{1}{2r}$ as $h\rightarrow0$,
so the left hand side of the above inequality tends to $\frac{1}{2r}D_{1}\varphi(re_{1})$.
Thus we must have $D_{i}\varphi(re_{1})\leq0$. On the other hand,
repeating the previous arguments, but instead selecting $y=re_{1}-he_{i}$
at the beginning, we can deduce the estimate
\[
-\frac{1}{h}D_{i}\varphi(re_{1})+\frac{1}{2}D_{ii}\varphi(re_{1})\leq\frac{S(h)}{h^{2}}D_{1}\varphi(re_{1}-he_{i})+M\frac{\left|S(h)\right|^{2}}{h^{2}}+\frac{o(\left|h\right|^{2})}{h^{2}},
\]
from which it follows that $D_{i}\varphi(re_{1})\geq0$. Thus $D_{i}\varphi(re_{1})=0$
and we may let $h\rightarrow0$ to obtain that $D_{ii}(re_{1})\leq\frac{1}{r}D_{1}\varphi(re_{1})$.
\end{proof}
\begin{thm}[Comparison principle]
\label{thm:comparison principle}Let $w$ and $v$ respectively be
bounded weak sub- and supersolutions to (\ref{eq:radial eq}) in $(0,R)$.
Assume that we have
\[
\limsup_{r\rightarrow R}w(r)\leq\liminf_{r\rightarrow R}v(r)
\]
Then $w\leq v$ a.e.\ in $(0,R)$.
\end{thm}
\begin{proof}
Let $\varepsilon>0$. Then there is $0<R^{\prime}<R$ such that $w-v-\varepsilon<0$
in $(R^{\prime},R)$. We set
\[
\varphi:=\max(w-v-\varepsilon,0).
\]
By Lemma \ref{lem:max is in H1q} we have $\varphi\in W^{1,q}(r^{d-1},(0,R))$
with
\[
\varphi^{\prime}=\begin{cases}
w^{\prime}-v^{\prime}, & \text{a.e.\ in }\left\{ w>v+\varepsilon\right\} ,\\
0, & \text{a.e.\ in }(0,R)\setminus\left\{ w>v+\varepsilon\right\} .
\end{cases}
\]
By Lemma \ref{lem:admissible test functions} we may use $\varphi$
as a test function in (\ref{eq:weak super}) for $w$ and $v$. This
yields the inequalities
\begin{align*}
\int_{\left\{ w>v+\varepsilon\right\} }\kappa\left|w^{\prime}\right|^{q-2}w^{\prime}(w^{\prime}-v^{\prime})r^{d-1}\d r\leq & \int_{0}^{R}\varphi fr^{d-1}\d r,\\
\int_{\left\{ w>v+\varepsilon\right\} }\kappa\left|v^{\prime}\right|^{q-2}v^{\prime}(w^{\prime}-v^{\prime})r^{d-1}\d r\geq & \int_{0}^{R}\varphi fr^{d-1}\d r.
\end{align*}
Subtracting the second inequality from the first we get
\[
\int_{\left\{ w>v+\varepsilon\right\} }\kappa(\left|w^{\prime}\right|^{q-2}w^{\prime}-\left|v^{\prime}\right|^{q-2}v^{\prime})(w^{\prime}-v^{\prime})r^{d-1}\d r\leq0.
\]
Since $\smash{(\left|a\right|^{q-2}a-\left|b\right|^{q-2}b)\left(a-b\right)\geq0}$
for all $a,b\in\mathbb{R},$ it follows that $\smash{w^{\prime}-v^{\prime}\equiv0}$
in $\left\{ w>v+\varepsilon\right\} $. Hence $\smash{\varphi^{\prime}\equiv0}$
a.e.\ in $(0,R)$. This implies that also $\varphi\equiv0$ a.e.\ in
$(0,R)$ since we have $\smash{\varphi\in W_{loc}^{1,q}(0,R)}$ and
$\varphi\equiv0$ in $\smash{(R^{\prime},R)}$. Consequently $w\leq v-\varepsilon$
a.e.\ in $(0,R)$ and letting $\varepsilon\rightarrow0$ finishes
the proof.
\end{proof}
\begin{lem}
\label{lem:comparison lemma}Let $v$ be a bounded weak supersolution
to (\ref{eq:radial eq}) in $(0,R)$. Let $I\Subset(0,R)$ be an interval
and suppose that $\phi\in C^{2}(\overline{I})$ satisfies
\begin{equation}
\int_{I}\left|\phi^{\prime}\right|^{q-2}\phi^{\prime}\varphi^{\prime}r^{d-1}-\varphi fr^{d-1}\d r\le0\label{eq:comp b 1}
\end{equation}
for all $\varphi\in C_{0}^{\infty}(I)$. Assume also that for all
$r_{0}\in\partial I$ we have
\[
\limsup_{r\rightarrow r_{0}}\phi(r)\leq\liminf_{r\rightarrow r_{0}}v(r).
\]
Then $\phi\leq v$ a.e.\ in $I$. 
\end{lem}
\begin{proof}
Since $v\in W^{1,q}(r^{d-1},(0,R))$ we have $v_{|I}\in W^{1,q}(I)$
with $(v_{|I})^{\prime}=v^{\prime}$ in $I$. Moreover, we have
\begin{equation}
\int_{0}^{R}\sa(v_{|I})^{\prime}\sa^{q-2}(v_{|I})^{\prime}\varphi^{\prime}r^{d-1}-\varphi fr^{d-1}\d r\geq0\label{eq:comp b 2}
\end{equation}
 for all $\varphi\in C_{0}^{\infty}(I)$. For $\varepsilon>0$, we
set
\[
\varphi:=(\phi-v_{|I}-\varepsilon)_{+}.
\]
Then $\varphi\in W^{1,q}(I)$ and $\supp\varphi\Subset I$. Thus we
may after approximation use $\varphi$ as a test function in (\ref{eq:comp b 1})
and (\ref{eq:comp b 2}). It then follows similarly as in the proof
of Theorem \ref{thm:comparison principle} that $\varphi\equiv0$
a.e.\ in $I$ and letting $\varepsilon\rightarrow0$ finishes the
proof.
\end{proof}

\section{Viscosity solutions are weak solutions}

We show that bounded radial viscosity supersolutions to (\ref{eq:normplap f})
are weak supersolutions to (\ref{eq:radial eq}). More precisely,
we prove the following theorem.
\begin{thm}
\label{thm:visc is weak}Let $u$ be a bounded radial viscosity supersolution
to (\ref{eq:normplap f}) in $B_{R}$. Then $v(r):=u(re_{1})$ is
a weak supersolution to (\ref{eq:radial eq}) in $(0,R)$.
\end{thm}
As a corollary of Theorem \ref{thm:visc is weak}, we obtain the uniqueness
of radial viscosity solutions to (\ref{eq:normplap f}). We also have
the following comparison result for radial super- and subsolutions.
However, the full uniqueness and comparison principle still remain
open as far as we know.
\begin{lem}
\label{lem:visc comp}Let $h,u\in C(B_{R})$ be bounded radial viscosity
sub- and supersolutions to (\ref{eq:normplap f}) in $B_{R}$, respectively.
Assume that for all $x_{0}\in\partial B_{R}$ it holds 
\[
\limsup_{x\rightarrow x_{0}}h(x)\leq\liminf_{x\rightarrow x_{0}}u(x).
\]
Then $h\leq u$ in $B_{R}$.
\end{lem}
\begin{proof}
By Theorem \ref{thm:visc is weak}, the functions $w(r):=h(re_{1})$
and $v(r):=u(re_{1})$ are weak sub- and supersolutions to (\ref{eq:radial eq})
in $(0,R)$, respectively. Hence by Theorem \ref{thm:comparison principle}
we have $w\leq v$ a.e.\ in $(0,R)$. It follows from continuity
that $h\leq u$ in $B_{R}$.
\end{proof}
\begin{cor}
Let $u,h\in C(\overline{B_{R}})$ be radial viscosity solutions to
(\ref{eq:normplap f}) in $B_{R}$ such that $u=h$ on $\partial B_{R}$.
Then $u=h$.
\end{cor}
One way to prove that viscosity solutions are weak solutions is by
using a comparison principle \cite{equivalence_plaplace}. As mentioned
however, full comparison principle for the equation (\ref{eq:normplap f})
is open and Lemma \ref{lem:visc comp} is not \textit{a priori} available.
Therefore we use the method developed by Julin and Juutinen \cite{newequivalence}.
The idea is to approximate a viscosity supersolution $u$ by its inf-convolution
\[
u_{\varepsilon}(x):=\inf_{y\in B_{R}}\left\{ u(y)+\frac{\left|x-y\right|^{\hat{q}}}{\hat{q}\varepsilon^{\hat{q}-1}}\right\} ,
\]
where $\varepsilon>0$ and $\hat{q}>2$ is a fixed constant so large
that $q-2+(\hat{q}-2)/(\hat{q}-1)>0$. Then $u_{\varepsilon}\rightarrow u$
pointwise in $B_{R}$ and it is standard to show that $u_{\varepsilon}$
is a viscosity supersolution to 
\begin{equation}
-\left|Du_{\varepsilon}\right|^{q-2}\Delta_{p}^{N}u_{\varepsilon}\geq f_{\varepsilon}(\left|x\right|)\quad\text{in }B_{R_{\varepsilon}},\label{eq:lala}
\end{equation}
where $f_{\varepsilon}(r):=\inf_{\left|r-s\right|\leq\rho(\varepsilon)}f(s)$,
$R_{\varepsilon}:=R-\rho(\varepsilon)$ and $\rho(\varepsilon)\rightarrow0$
as $\varepsilon\rightarrow0$. Moreover, $u_{\varepsilon}$ is semi-concave
by definition and thus twice differentiable almost everywhere by Alexandrov's
theorem (see e.g. \cite[p273]{measuretheoryevans}). Hence $u_{\varepsilon}$
satisfies the equation (\ref{eq:lala}) pointwise almost everywhere.
Since $u_{\varepsilon}$ is still radial, we can perform a radial
transformation on (\ref{eq:lala}) to obtain after mollification arguments
that $v_{\varepsilon}(r):=u_{\varepsilon}(re_{1})$ is a weak supersolution
to $-\kappa\Delta_{q}^{d}v_{\varepsilon}=f_{\varepsilon}$ in $(0,R_{\varepsilon})$.
Caccioppoli's estimate then implies that $v_{\varepsilon}$ converges
to $v$ in the weighted Sobolev space up to a subsequence and we obtain
that $v_{\varepsilon}$ is a weak supersolution. 

Before beginning the proof of Theorem \ref{thm:visc is weak}, we
collect some well known properties of inf-convolution in the following
lemma (see e.g.\ \cite{userguide,nikos,newequivalence}).
\begin{lem}
\label{lem:inf conv properties}Assume that $u:B_{R}\rightarrow\mathbb{R}$
is bounded and lower semicontinuous. Then the inf-convolution $u_{\varepsilon}$
has the following properties.
\begin{enumerate}
\item We have $u_{\varepsilon}\leq u$ and $u_{\varepsilon}\rightarrow u$
pointwise in $B_{R}$ as $\varepsilon\rightarrow0$.
\item There exists $\rho(\varepsilon)>0$ such that
\[
u_{\varepsilon}(x)=\inf_{y\in B_{\rho(\varepsilon)}(x)\cap B_{R}}\left\{ u(y)+\frac{1}{q\varepsilon^{\hat{q}-1}}\left|x-y\right|^{\hat{q}}\right\} 
\]
and $\rho(\varepsilon)\rightarrow0$ as $\varepsilon\rightarrow0$.
In fact we can choose $\rho(\varepsilon)=\left(q\varepsilon^{\hat{q}-1}\mathrm{osc}_{B_{R}}\thinspace u\right)^{\frac{1}{\hat{q}}}$.
\item Denote $R_{\varepsilon}:=R-\rho(\varepsilon)$. Then $u_{\varepsilon}$
is semi-concave in $B_{R_{\varepsilon}}$. Moreover, for any $x\in B_{R_{\varepsilon}}$
there is $x_{\varepsilon}\in\overline{B_{\rho(\varepsilon)}(x)}$
such that $u_{\varepsilon}(x)=u(x_{\varepsilon})+\frac{1}{\hat{q}\varepsilon^{\hat{q}-1}}\left|x-x_{\varepsilon}\right|^{\hat{q}}.$
\item If $u_{\varepsilon}$ is twice differentiable at $x\in B_{R_{\varepsilon}}$,
then
\begin{align}
Du_{\varepsilon}(x)= & (x-x_{\varepsilon})\frac{\left|x-x_{\varepsilon}\right|^{\hat{q}-2}}{\varepsilon^{\hat{q}-1}},\label{eq:inf conv diff}\\
D^{2}u_{\varepsilon}(x)\leq & (\hat{q}-1)\frac{\left|x-x_{\varepsilon}\right|^{\hat{q}-2}}{\varepsilon^{\hat{q}-1}}I.\nonumber 
\end{align}
\end{enumerate}
\end{lem}
\begin{rem*}
Observe that if $u$ is radial, then so is $u_{\varepsilon}$. Moreover,
if we set $v(r):=u_{\varepsilon}(re_{1})$ and assume that $v$ is
twice differentiable at $r\in(0,R_{\varepsilon})$, then by (iv) of
Lemma \ref{lem:inf conv properties} we have
\begin{align}
v_{\varepsilon}^{\prime}(r)= & (r-r_{\varepsilon})\frac{\left|r-r_{\varepsilon}\right|}{\varepsilon^{\hat{q}-1}}^{\hat{q}-2},\label{eq:inf conv radial diff-1}\\
v_{\varepsilon}^{\prime\prime}(r)\leq & \frac{\hat{q}-1}{\varepsilon}\left|v_{\varepsilon}^{\prime}(r)\right|^{\frac{\hat{q}-2}{\hat{q}-1}},\label{eq:inf conv radial hess-1}
\end{align}
where $r_{\varepsilon}\in(r-\rho(\varepsilon),r+\rho(\varepsilon))$. 
\end{rem*}
\begin{lem}
\label{lem:infconv is visc super}Assume that $u$ is a bounded viscosity
supersolution to (\ref{eq:normplap f}) in $B_{R}$. Then $u_{\varepsilon}$
is a viscosity supersolution to (\ref{eq:lala}) in $B_{R_{\varepsilon}}$.
\end{lem}
\begin{proof}
Suppose that $\varphi\in C^{2}$ touches $u_{\varepsilon}$ from below
at $x\in B_{R_{\varepsilon}}$ and that $D\varphi(y)\not=0$ when
$y\not=x$. Let $x_{\varepsilon}$ be as in (iii) of Lemma \ref{lem:inf conv properties}.
Then 
\begin{align}
\varphi(x)= & u_{\varepsilon}(x)=u(x_{\varepsilon})+\frac{1}{\hat{q}\varepsilon^{\hat{q}-1}}\left|x-x_{\varepsilon}\right|^{\hat{q}},\label{eq:infconv is visc super 1}\\
\varphi(y)\leq & u_{\varepsilon}(y)\leq u(z)+\frac{1}{\hat{q}\varepsilon^{\hat{q}-1}}\left|y-z\right|^{\hat{q}}\quad\text{for all }y,z\in B_{R}.\label{eq:infconv is visc super 2}
\end{align}
Set 
\[
\psi(z)=\varphi(z+x-x_{\varepsilon})-\frac{\left|x-x_{\varepsilon}\right|^{\hat{q}}}{\hat{q}\varepsilon^{\hat{q}-1}}.
\]
It follows from (\ref{eq:infconv is visc super 1}) and (\ref{eq:infconv is visc super 2})
that $\psi$ touches $u$ from below at $x_{\varepsilon}$. Therefore,
since $u$ is a viscosity supersolution to (\ref{eq:normplap f}),
we have
\begin{align*}
0\leq & \limsup_{x\not=z\rightarrow x_{\varepsilon}}\big(-\left|D\psi(z)\right|^{q-2}\Delta_{p}^{N}\psi(z)\big)-f(\sa x_{\varepsilon}\sa)\\
\leq & \limsup_{x\not=y\rightarrow x}\big(-\left|D\varphi(y)\right|^{q-2}\Delta_{p}^{N}\varphi(y)\big)-f_{\varepsilon}(\left|x\right|),
\end{align*}
where we used that $\left|x-x_{\varepsilon}\right|\leq\rho(\varepsilon)$
and $f_{\varepsilon}(r)=\inf_{\left|r-s\right|\leq\rho(\varepsilon)}f(s)$.
Consequently $u_{\varepsilon}$ is a viscosity supersolution to (\ref{eq:lala}).
\end{proof}
Next we combine the previous lemma with the radial transformation
of (\ref{eq:lala}).
\begin{lem}
\label{lem:inf conv pointwise}Assume that $u$ is a bounded radial
viscosity supersolution to (\ref{eq:normplap f}) in $B_{R}$. Set
$v_{\varepsilon}(r):=u_{\varepsilon}(re_{1})$ and assume that $v_{\varepsilon}$
is twice differentiable at $r\in(0,R_{\varepsilon})$. Then, if $q>2$
or $v_{\varepsilon}^{\prime}(r)\not=0$, we have
\begin{equation}
-\kappa\left|v_{\varepsilon}^{\prime}(r)\right|^{q-2}\big((q-1)v_{\varepsilon}^{\prime\prime}(r)+\frac{d-1}{r}v_{\varepsilon}^{\prime}(r)\big)-f_{\varepsilon}(r)\geq0.\label{eq:inf conv pointwise}
\end{equation}
Moreover, if $1<q\leq2$ with $v_{\varepsilon}^{\prime}(r)=0$, then
we have $f_{\varepsilon}(r)\leq0.$
\end{lem}
\begin{proof}
Consider first the case $q>2$ or $v_{\varepsilon}^{\prime}(r)\not=0$.
Since $u_{\varepsilon}$ is twice differentiable at $re_{1}$, it
follows from the definition of viscosity supersolutions that
\begin{equation}
-\left|Du_{\varepsilon}(re_{1})\right|^{q-2}(\Delta u_{\varepsilon}(re_{1})+(p-2)\Delta_{\infty}^{N}u_{\varepsilon}(re_{1}))-f_{\varepsilon}(r)\geq0,\label{eq:inf conv pointwise 1}
\end{equation}
where 
\begin{equation}
\Delta_{\infty}^{N}u_{\varepsilon}=\left|Du_{\varepsilon}\right|^{-2}\sum_{i,j=1}^{N}D_{ij}u_{\varepsilon}D_{i}u_{\varepsilon}D_{j}u_{\varepsilon}.\label{eq:norm inf lap}
\end{equation}
 Moreover, we have
\[
Du_{\varepsilon}(re_{1})=e_{1}v_{\varepsilon}^{\prime}(r)\quad\text{and\quad}D^{2}u_{\varepsilon}(re_{1})=e_{1}\otimes e_{1}v_{\varepsilon}^{\prime\prime}(r)+\frac{1}{r}\big(I-e_{1}\otimes e_{1}\big)v_{\varepsilon}^{\prime}(r).
\]
It is now straightforward to compute that
\[
\Delta u_{\varepsilon}(re_{1})=\tr D^{2}u_{\varepsilon}(re_{1})=v_{\varepsilon}^{\prime\prime}(r)+\frac{N-1}{r}v_{\varepsilon}^{\prime}(r)
\]
and using (\ref{eq:norm inf lap}) 
\[
\Delta_{\infty}^{N}u_{\varepsilon}(re_{1})=v_{\varepsilon}^{\prime\prime}(r).
\]
Combining these with (\ref{eq:inf conv pointwise 1}) and recalling
that $d-1=(N-1)(q-1)/(p-1)$, $\kappa=(p-1)/(q-1)$, we obtain (\ref{eq:inf conv pointwise}).

Consider then the case $1<q\leq2$ and $v_{\varepsilon}^{\prime}(r)=0$.
Denote $x:=re_{1}$. Then $Du(x)=0$ and so by (\ref{eq:inf conv diff})
we have $x_{\varepsilon}=x$. Therefore by the definition of inf-convolution
\[
u(y)+\frac{\left|x-y\right|^{\hat{q}}}{\hat{q}\varepsilon^{\hat{q}-1}}\geq u_{\varepsilon}(x)=u(x)\quad\text{for all }y\in B_{R}.
\]
Rearranging the terms, we find that
\[
\phi(y):=u(x)-\frac{\left|x-y\right|^{\hat{q}}}{\hat{q}\varepsilon^{\hat{q}-1}}\leq u(y)\quad\text{for all }y\in B_{R}.
\]
In other words, the function $\phi$ touches $u$ from below at $x$.
Since $u$ is a viscosity supersolution and $D\phi(y)\not=0$ when
$y\not=x$, it follows that
\[
\limsup_{y\rightarrow x,y\not=x}\left(\left|D\phi(y)\right|^{q-2}\Delta_{p}^{N}\phi(y)-f(\left|y\right|)\right)\geq0.
\]
This implies that $-f(\left|x\right|)=-f(r)\geq0$ since $\left|D\phi(y)\right|^{q-2}\Delta_{p}^{N}\phi(y)\rightarrow0$
as $y\rightarrow x$. Indeed, we have
\begin{align*}
\left|D\phi(y)\right|^{q-2}\left|\Delta_{p}^{N}\phi(y)\right|\leq & C(q,\hat{q},\varepsilon)\left|y-x\right|^{(q-2)(\hat{q}-1)}(N+\left|p-2\right|)||D^{2}\phi(y)||\\
\leq & C(q,\hat{q},p,N,\varepsilon)\left|y-x\right|^{(q-2)(\hat{q}-1)+\hat{q}-2},
\end{align*}
where $(q-2)(\hat{q}-1)+\hat{q}-2>0$ by definition of $\hat{q}$.
\end{proof}
Next we show that the inf-convolution is a weak supersolution to $-\kappa\Delta_{q}^{d}u_{\varepsilon}=f_{\varepsilon}$
in $(0,R_{\varepsilon})$. We consider the case $q>2$ first.
\begin{lem}
\label{lem:inf conv is weak q > 2}Let $q>2$. Assume that $u$ is
a bounded radial viscosity supersolution to (\ref{eq:normplap f})
in $B_{R}$. Then the function $v_{\varepsilon}(r):=u_{\varepsilon}(re_{1})$
is a weak supersolution to $-\kappa\Delta_{q}^{d}u=f_{\varepsilon}$
in $(0,R_{\varepsilon})$.
\end{lem}
\begin{proof}
Since $u_{\varepsilon}$ is semi-concave in $B_{R_{\varepsilon}}$,
it is also locally Lipschitz continuous there \cite[p267]{measuretheoryevans}.
Consequently we have $v_{\varepsilon}\in W^{1,q}(r^{d-1},(0,R^{\prime}))$
for all $R^{\prime}\in(0,R_{\varepsilon})$ by Lemma \ref{lem:sobolev is in H1q}.
Observe then that since $\phi(x):=u_{\varepsilon}(x)-C(\hat{q},\varepsilon,u)\left|x\right|^{2}$
is concave in $B_{R_{\varepsilon}}$, it is twice differentiable almost
everywhere by Alexandrov's theorem. Moreover, the proof of Alexandrov's
theorem in \cite[p273]{measuretheoryevans} establishes that we can
approximate $\phi$ by smooth concave radial functions $\phi_{j}$
with the standard mollification. Therefore, by setting $u_{\varepsilon,j}(x):=\phi_{j}(x)+C(\hat{q},\varepsilon,u)\left|x\right|^{2}$,
the following pointwise limits hold almost everywhere in $B_{R_{\varepsilon}}$
\[
u_{\varepsilon,j}\rightarrow u_{\varepsilon},\quad Du_{\varepsilon,j}\rightarrow Du_{\varepsilon}\quad\text{and\quad}D^{2}u_{\varepsilon,j}\rightarrow D^{2}u_{\varepsilon}.
\]
Thus, since $u_{\varepsilon}$ is radial, setting $v_{\varepsilon,j}(r):=u_{\varepsilon,j}(re_{1})$
we have 
\[
v_{\varepsilon,j}\rightarrow v_{\varepsilon},\quad v_{\varepsilon,j}^{\prime}\rightarrow v_{\varepsilon}^{\prime}\quad\text{and\quad}v_{\varepsilon,j}^{\prime\prime}\rightarrow v_{\varepsilon}^{\prime\prime}
\]
almost everywhere in $(0,R_{\varepsilon})$. \begin{note2} (since
$v_{\varepsilon}^{\prime}(r)=e_{1}Du_{\varepsilon}(re_{1})$ and $v_{\varepsilon}^{\prime\prime}(r)=e_{1}^{\prime}D^{2}u_{\varepsilon}(re_{1})e_{1}$)
\end{note2} Since $v_{\varepsilon,j}$ is smooth and $q>2$, a direct
calculation yields for $r\in(0,R_{\varepsilon})$
\begin{align}
-\kappa\sa v_{\varepsilon,j}^{\prime}\sa^{q-2}\big((q-1)v_{\varepsilon,j}^{\prime\prime}+\frac{d-1}{r}v_{\varepsilon,j}^{\prime}\big)r^{d-1}= & -\kappa(\sa v_{\varepsilon,j}^{\prime}\sa^{q-2}v_{\varepsilon,j}^{\prime}r^{d-1})^{\prime}.\label{eq:inf conv is weak q>2 1}
\end{align}
Fix a non-negative $\varphi\in C_{0}^{\infty}(-R_{\varepsilon},R_{\varepsilon})$.
Then, integrating by parts we find for $h>0$
\begin{align*}
 & \int_{h}^{R}-\varphi\kappa(\sa v_{\varepsilon,j}^{\prime}\sa^{q-2}v_{\varepsilon,j}^{\prime}r^{d-1})^{\prime}\d r\\
 & =\int_{h}^{R}\kappa\sa v_{\varepsilon,j}^{\prime}\sa^{q-2}v_{\varepsilon,j}^{\prime}r^{d-1}\varphi^{\prime}\d r+\varphi(h)\kappa\sa v_{\varepsilon,j}^{\prime}(h)\sa^{q-2}v_{\varepsilon,j}^{\prime}(h)h^{d-1}.
\end{align*}
Combining this with (\ref{eq:inf conv is weak q>2 1}), letting $h\rightarrow0$
and subtracting $\int_{0}^{R}\varphi f_{\varepsilon}r^{d-1}\d r$
from both sides, we obtain
\begin{align*}
\int_{0}^{R} & -\kappa\varphi\sa v_{\varepsilon,j}^{\prime}\sa^{q-2}\big((q-1)v_{\varepsilon,j}^{\prime\prime}+\frac{d-1}{r}v_{\varepsilon,j}^{\prime}\big)r^{d-1}-\varphi f_{\varepsilon}r^{d-1}\d r\\
 & =\int_{0}^{R}\kappa\sa v_{\varepsilon,j}^{\prime}\sa^{q-2}v_{\varepsilon,j}^{\prime}\varphi^{\prime}r^{d-1}-\varphi f_{\varepsilon}r^{d-1}\d r.
\end{align*}
Since $v_{\varepsilon,j}$ is Lipschitz continuous, we have $\smash{M:=\sup_{j}\sa\sa v_{\varepsilon,j}^{\prime}\sa\sa_{L^{\infty}(\supp\varphi)}<\infty}$.
Thus we may let $j\rightarrow\infty$ in the above inequality and
apply the dominated convergence theorem at the right hand side to
obtain
\begin{align}
 & \liminf_{j\rightarrow\infty}\int_{0}^{R}-\kappa\varphi\sa v_{\varepsilon,j}^{\prime}\sa^{q-2}\big((q-1)v_{\varepsilon,j}^{\prime\prime}+\frac{d-1}{r}v_{\varepsilon,j}^{\prime}\big)r^{d-1}-\varphi f_{\varepsilon}r^{d-1}\d r\nonumber \\
 & \ \leq\int_{0}^{R}\kappa\sa v_{\varepsilon}^{\prime}\sa^{q-2}v_{\varepsilon}^{\prime}\varphi^{\prime}r^{d-1}-\varphi f_{\varepsilon}r^{d-1}\d r.\label{eq:inf conv is weak q>2 2}
\end{align}
It now suffices to show that the left-hand side is non-negative to
finish the proof. Observe that $\smash{v_{\varepsilon,j}^{\prime\prime}\leq C(\hat{q},\varepsilon,u)}$
since $\phi_{j}$ is concave. Thus 
\begin{align*}
-\sa v_{\varepsilon,j}^{\prime}\sa^{q-2}\big((q-1)v_{\varepsilon,j}^{\prime\prime}+\frac{d-1}{r}v_{\varepsilon,j}^{\prime})\geq & -M^{q-2}\big(C(q,\hat{q},\varepsilon,u)+\frac{d-1}{r}M\big).
\end{align*}
Since $d-2>-1$, it follows from the above inequality that the integral
at the left hand side of (\ref{eq:inf conv is weak q>2 2}) has an
integrable lower bound. Hence by Fatou's lemma
\begin{align*}
 & \liminf_{j\rightarrow\infty}\int_{0}^{R}-\kappa\varphi\sa v_{\varepsilon,j}^{\prime}\sa^{q-2}\big((q-1)v_{\varepsilon,j}^{\prime\prime}+\frac{d-1}{r}v_{\varepsilon,j}^{\prime}\big)r^{d-1}-\varphi f_{\varepsilon}r^{d-1}\d r\\
 & \ \geq\int_{0}^{R}-\kappa\varphi\sa v_{\varepsilon}^{\prime}\sa^{q-2}\big((q-1)v_{\varepsilon}^{\prime\prime}+\frac{d-1}{r}v_{\varepsilon}^{\prime}\big)r^{d-1}-\varphi f_{\varepsilon}r^{d-1}\d r\geq0,
\end{align*}
where the last inequality follows from Lemma \ref{lem:inf conv pointwise}.
\end{proof}
Next we consider the case $1<q\leq2$. We need an additional regularization
step because of the singularity of $\left|Du\right|^{q-2}\Delta_{p}^{N}u$
at the points where the gradient vanishes.
\begin{lem}
\label{lem:inf-conv is weak}Let $1<q\leq2$. Assume that $u$ is
a bounded radial viscosity supersolution to (\ref{eq:normplap f})
in $B_{R}$. Then the function $v_{\varepsilon}(r):$$=u_{\varepsilon}(re_{1})$
is a weak supersolution to $-\kappa\Delta_{q}^{d}u\geq f_{\varepsilon}$
in $(0,R_{\varepsilon})$.
\end{lem}
\begin{proof}
\textbf{(Step 1)} We define the smooth semi-concave functions $v_{\varepsilon,j}$
exactly as in the proof of Lemma \ref{lem:inf conv is weak q > 2}.
Then again
\[
v_{\varepsilon,j}\rightarrow v_{\varepsilon},\quad v_{\varepsilon,j}^{\prime}\rightarrow v_{\varepsilon}^{\prime}\quad\text{and\quad}v_{\varepsilon,j}^{\prime\prime}\rightarrow v_{\varepsilon}^{\prime\prime}
\]
almost everywhere in $(0,R_{\varepsilon})$. Let $\delta>0$. We regularize
the radial transformation of equation (\ref{eq:normplap f}) by considering
the following term
\[
G_{\delta}(v):=-\kappa(\left|v^{\prime}\right|^{2}+\delta)^{\frac{q-2}{2}}\Big(\big(1+(q-2)\frac{\left|v^{\prime}\right|^{2}}{\left|v^{\prime}\right|^{2}+\delta}\big)v^{\prime\prime}+\frac{d-1}{r}v^{\prime}\Big).
\]
Since $v_{\varepsilon,j}$ is smooth, a direct calculation yields
for $r\in(0,R_{\varepsilon})$
\begin{equation}
G_{\delta}(v_{\varepsilon,j})r^{d-1}=-\kappa\big((\sa v_{\varepsilon,j}^{\prime}\sa^{2}+\delta)^{\frac{q-2}{2}}v_{\varepsilon,j}^{\prime}r^{d-1}\big)^{\prime}.\label{eq:direct calc}
\end{equation}
Fix a non-negative $\varphi\in C_{0}^{\infty}(-R_{\varepsilon},R_{\varepsilon})$.
Then, integrating by parts we have for $h>0$
\begin{align*}
 & \int_{h}^{R}-\kappa\varphi\big((\sa v_{\varepsilon,j}^{\prime}\sa^{2}+\delta)^{\frac{q-2}{2}}v_{\varepsilon,j}^{\prime}r^{d-1}\big)^{\prime}\d r\\
 & =\int_{h}^{R}\kappa(\sa v_{\varepsilon,j}^{\prime}\sa^{2}+\delta)^{\frac{q-2}{2}}v_{\varepsilon,j}^{\prime}r^{d-1}\varphi^{\prime}\d r+\varphi(h)\kappa(\sa v_{\varepsilon,j}^{\prime}(h)\sa^{2}+\delta)^{\frac{q-2}{2}}v_{\varepsilon,j}^{\prime}(h)h^{d-1}.
\end{align*}
Combining this with (\ref{eq:direct calc}), letting $h\rightarrow0$
and subtracting $\int_{0}^{R}\varphi f_{\varepsilon}r^{d-1}\d r$
from both sides we obtain
\[
\int_{0}^{R}\varphi G_{\delta}(v_{\varepsilon,j})r^{d-1}-\varphi f_{\varepsilon}r^{d-1}\d r=\int_{0}^{R}\kappa(\sa v_{\varepsilon,j}^{\prime}\sa^{2}+\delta)^{\frac{q-2}{2}}v_{\varepsilon,j}^{\prime}\varphi^{\prime}r^{d-1}-\varphi f_{\varepsilon}r^{d-1}\d r.
\]
This implies that
\begin{align}
 & \liminf_{j\rightarrow\infty}\int_{0}^{R}\varphi G_{\delta}(v_{\varepsilon,j})r^{d-1}-\varphi f_{\varepsilon}r^{d-1}\d r\nonumber \\
 & \ \leq\lim_{j\rightarrow\infty}\int_{0}^{R}\kappa(\sa v_{\varepsilon,j}^{\prime}\sa^{2}+\delta)^{\frac{q-2}{2}}v_{\varepsilon,j}^{\prime}\varphi^{\prime}r^{d-1}-\varphi f_{\varepsilon}r^{d-1}\d r.\label{eq:preauxiliary}
\end{align}
We intend to apply Fatou's lemma at the left-hand side and the dominated
convergence theorem at the right-hand side. Since $v_{\varepsilon}$
is Lipschitz continuous, we have $M:=\sup_{j}||v_{\varepsilon,j}^{\prime}||_{L^{\infty}(\supp\varphi)}<\infty$,
which justifies the use of the dominated convergence theorem. Observe
then that $v_{\varepsilon,j}^{\prime\prime}\leq C(\hat{q},\varepsilon,u)$
by semi-concavity. Hence
\begin{align*}
G_{\delta}(v_{\varepsilon,j})= & -\kappa(\sa v_{\varepsilon,j}^{\prime}\sa^{2}+\delta)^{\frac{q-2}{2}}\Big(\big(1+(q-2)\frac{\sa v_{\varepsilon,j}\sa^{2}}{\sa v_{\varepsilon,j}\sa^{2}+\delta}\big)v_{\varepsilon,j}^{\prime\prime}+\frac{d-1}{r}v_{\varepsilon,j}^{\prime}\Big)\\
\geq & -\kappa\delta^{\frac{q-2}{2}}(C(q,\hat{q},\varepsilon,u)+\frac{d-1}{r}M).
\end{align*}
Since $d-2>-1$, it follows from the above estimate that the integrand
at the left-hand side of (\ref{eq:preauxiliary}) has an integrable
lower bound independent of $j$. Thus
\begin{align}
\int_{0}^{R}\varphi\left(G_{\delta}(v_{\varepsilon})-f_{\varepsilon}\right)r^{d-1}\d r & \leq\int_{0}^{R}\kappa(\left|v_{\varepsilon}^{\prime}\right|^{2}+\delta)^{\frac{q-2}{2}}v_{\varepsilon}^{\prime}\varphi^{\prime}r^{d-1}-\varphi f_{\varepsilon}r^{d-1}\d r.\label{eq:auxiliary}
\end{align}

\textbf{(Step 2)} We let $\delta\rightarrow0$ in the auxiliary inequality
(\ref{eq:auxiliary}) and obtain
\begin{align}
 & \liminf_{\delta\rightarrow0}\int_{0}^{R}\varphi(G_{\delta}(v_{\varepsilon})-f_{\varepsilon})r^{d-1}dr\nonumber \\
 & \ \ \leq\lim_{\delta\rightarrow0}\int_{0}^{R}\kappa(\left|v_{\varepsilon}^{\prime}\right|^{2}+\delta)^{\frac{q-2}{2}}v_{\varepsilon}^{\prime}\varphi^{\prime}r^{d-1}-\varphi f_{\varepsilon}r^{d-1}dr\nonumber \\
 & \ \ =\int_{0}^{R}\kappa\left|v_{\varepsilon}^{\prime}\right|^{q-2}v_{\varepsilon}^{\prime}\varphi^{\prime}r^{d-1}-\varphi f_{\varepsilon}r^{d-1}\d r,\label{eq:auxiliary2}
\end{align}
where the use of the dominated convergence theorem was justified since
$v_{\varepsilon}$ is Lipschitz continuous. It now suffices to show
that the left-hand side of (\ref{eq:auxiliary2}) is non-negative
to finish the proof. By (\ref{eq:inf conv radial hess-1}) we have
\begin{equation}
v_{\varepsilon}^{\prime\prime}\leq\frac{\hat{q}-1}{\varepsilon}\left|v_{\varepsilon}^{\prime}\right|^{\frac{\hat{q}-2}{\hat{q}-1}}\label{eq:inf conv 2nd diff est}
\end{equation}
almost everywhere in $(0,R_{\varepsilon})$. Hence, when $v_{\varepsilon}^{\prime}\not=0$,
it holds that
\begin{align*}
G_{\delta}(v_{\varepsilon}) & =-\kappa(\left|v_{\varepsilon}^{\prime}\right|^{2}+\delta)^{\frac{q-2}{2}}\Big(\big(1+(q-2)\frac{\left|v_{\varepsilon}^{\prime}\right|^{2}}{\left|v_{\varepsilon}^{\prime}\right|^{2}+\delta}\big)v_{\varepsilon}^{\prime\prime}+\frac{d-1}{r}v_{\varepsilon}^{\prime}\Big)\\
 & \geq-\kappa\left|v_{\varepsilon}^{\prime}\right|^{q-2}(C(q,\hat{q},\varepsilon)\left|v_{\varepsilon}^{\prime}\right|^{\frac{\hat{q}-2}{\hat{q}-1}}+\frac{d-1}{r}\left|v_{\varepsilon}^{\prime}\right|)\\
 & =-\kappa(C(q,\hat{q},\varepsilon)\left|v_{\varepsilon}^{\prime}\right|^{q-2+\frac{\hat{q}-2}{\hat{q}-1}}+\frac{d-1}{r}\left|v_{\varepsilon}^{\prime}\right|^{q-1}),
\end{align*}
where $q-2+\frac{\hat{q}-2}{\hat{q}-1}\geq0$ by definition of $\hat{q}$.
Moreover, when $v_{\varepsilon}^{\prime}=0$, we have $G_{\delta}(v_{\varepsilon})\geq0$
directly by (\ref{eq:inf conv 2nd diff est}). Since $v_{\varepsilon}$
is Lipschitz continuous in the support of $\varphi$, these estimates
imply that the integrand at the left-hand side of (\ref{eq:auxiliary2})
has an integrable lower bound independent of $\delta$. Thus by Fatou's
lemma
\begin{align}
\liminf_{\delta\rightarrow0} & \int_{0}^{R}\varphi\left(G_{\delta}(v_{\varepsilon})-f_{\varepsilon}\right)r^{d-1}\d r\nonumber \\
\geq & \int_{0}^{R}\liminf_{\delta\rightarrow0}\varphi\left(G_{\delta}(v_{\varepsilon})-f_{\varepsilon}\right)r^{d-1}\d r\nonumber \\
= & \int_{\left\{ v_{\varepsilon}^{\prime}\not=0\right\} }\varphi\Big(-\kappa\left|v_{\varepsilon}^{\prime}\right|^{q-2}\big((q-1)v_{\varepsilon}^{\prime\prime}+\frac{d-1}{r}v_{\varepsilon}^{\prime}\big)-f_{\varepsilon}\Big)r^{d-1}\d r\nonumber \\
 & +\int_{\left\{ v_{\varepsilon}^{\prime}=0\right\} }\liminf_{\delta\rightarrow0}\varphi(-\kappa\delta^{\frac{q-2}{2}}v_{\varepsilon}^{\prime\prime}-f_{\varepsilon})r^{d-1}\d r\nonumber \\
=: & A_{1}+A_{2}.\label{eq:liminf est}
\end{align}
It follows directly from Lemma \ref{lem:inf conv pointwise} that
$A_{1}\geq0$. Moreover, if $r\in\left\{ v_{\varepsilon}^{\prime}=0\right\} $,
then Lemma \ref{lem:inf conv pointwise} implies that $f_{\varepsilon}(r)\leq0$
and inequality (\ref{eq:inf conv 2nd diff est}) reads as $v_{\varepsilon}^{\prime\prime}(r)\leq0$.
Hence also $A_{2}\geq0$. Combining (\ref{eq:auxiliary2}) and (\ref{eq:liminf est})
we have thus established the desired inequality.
\end{proof}
We use the following Caccioppoli's estimate to show that the sequence
$v_{\varepsilon}$ is bounded in the weighted Sobolev space.
\begin{lem}[Caccioppoli's estimate]
\label{lem:Caccioppoli} Let $v$ be a bounded weak supersolution
to (\ref{eq:radial eq}) in $(0,R)$. Suppose moreover that $v$ is
Lipschitz continuous in $(0,R^{\prime})$ for any $R^{\prime}\in(0,R)$.
Then for any non-negative $\xi\in C_{0}^{\infty}(-R,R)$ we have
\[
\int_{0}^{R}\left|v^{\prime}\right|^{q}\xi^{q}r^{d-1}\d r\leq C\int_{0}^{R}\left(\left|\xi^{\prime}\right|^{q}+\xi^{q}\left|f\right|\right)r^{d-1}\d r,
\]
where $C=C(\kappa,q,M)$ and $M=\left\Vert v\right\Vert _{L^{\infty}((0,R)\cap\supp\xi)}$.
\end{lem}
\begin{proof}
Since $\xi\in C_{0}^{\infty}(-R,R)$, we can use $\varphi:=(M-v)\xi^{q}$
as a test function by Lemma \ref{lem:admissible test functions}.
This yields
\begin{align*}
0\leq & \int_{0}^{R}\kappa\left|v^{\prime}\right|^{q-2}v^{\prime}(-v^{\prime}\xi^{q}+(M-v)q\xi^{\prime}\xi^{q-1})r^{d-1}-(M-v)\xi^{q}fr^{d-1}\d r.
\end{align*}
Rearranging the terms and using that $(M-v)\leq2M$, we obtain 
\begin{equation}
\int_{0}^{R}\kappa\left|v^{\prime}\right|^{q}\xi^{q}r^{d-1}\d r\leq2M\int_{0}^{R}q\left|v^{\prime}\right|^{q-1}\left|\xi^{\prime}\right|\xi^{q-1}r^{d-1}+\xi^{q}\left|f\right|r^{d-1}\d r.\label{eq:caccioppoli 1}
\end{equation}
By Young's inequality, we have for any $\epsilon>0$
\[
q\left|v^{\prime}\right|^{q-1}\left|\xi^{\prime}\right|\xi^{q-1}r^{d-1}\leq\epsilon\left|v^{\prime}\right|^{q}\xi^{q}r^{d-1}+C(q,\epsilon)\left|\xi^{\prime}\right|^{q}r^{d-1}.
\]
Applying this to (\ref{eq:caccioppoli 1}), taking small enough $\epsilon>0$
and absorbing the term with $v^{\prime}$ to the left-hand side, we
obtain the desired estimate. Absorbing the term is justified as it
is finite by the Lipschitz continuity of $v$.
\end{proof}
It now remains to use the Caccioppoli's estimate to obtain a subsequence
of $\smash{v_{\varepsilon}}$ that converges to $v$ in the weighted
Sobolev space. Then we can pass to the limit to see that $v$ is a
weak supersolution to (\ref{eq:radial eq}) in $(0,R)$.
\begin{proof}[Proof of Theorem \ref{thm:visc is weak}]
 Set $v_{\varepsilon}(r):=u_{\varepsilon}(re_{1})$ and let $0<R^{\prime\prime}<R$.
We start by showing that $v_{\varepsilon}\rightarrow v$ in $W^{1,q}(r^{d-1},(0,R^{\prime\prime}))$.
By assuming that $\varepsilon$ is small enough, we find $R^{\prime}$
such that
\[
R^{\prime\prime}<R^{\prime}<R_{\varepsilon}<R.
\]
Since by Lemmas \ref{lem:inf conv is weak q > 2} and \ref{lem:inf-conv is weak}
the function $v_{\varepsilon}$ is a weak supersolution to $-\kappa\Delta_{q}^{d}v_{\varepsilon}\geq f_{\varepsilon}$
in $(0,R_{\varepsilon})$, Lemma \ref{lem:Caccioppoli} implies that
$v_{\varepsilon}^{\prime}$ is bounded in $L^{q}(r^{d-1},(0,R^{\prime}))$.
Thus by Lemma \ref{lem:H1q conv} we have $v\in W^{1,q}(r^{d-1},(0,R^{\prime}))$
and $v_{\varepsilon}^{\prime}\rightarrow v^{\prime}$ weakly in $L^{q}(r^{d-1},(0,R^{\prime}))$
up to a subsequence. We set
\[
\varphi:=(v-v_{\varepsilon})\xi^{q},
\]
where $\xi\in C_{0}^{\infty}(-R^{\prime},R^{\prime})$ is a non-negative
cut-off function such that $\xi\equiv1$ in $(0,R^{\prime\prime})$.
Using $\varphi$ as a test function in the weak formulation of $-\kappa\Delta_{q}^{d}v_{\varepsilon}\geq f_{\varepsilon}$
we obtain
\[
0\leq\int_{0}^{R^{\prime}}\kappa\sa v_{\varepsilon}^{\prime}\sa^{q-2}v_{\varepsilon}^{\prime}\big((v^{\prime}-v_{\varepsilon}^{\prime})\xi^{q}+q\xi^{\prime}\xi^{q-1}(v-v_{\varepsilon})\big)r^{d-1}-(v-v_{\varepsilon})\xi^{q}f_{\varepsilon}r^{d-1}\d r.
\]
Rearranging the terms and adding $\int_{0}^{R^{\prime}}\kappa\sa v^{\prime}\sa^{q-2}v^{\prime}(v^{\prime}-v_{\varepsilon}^{\prime})\xi^{q}r^{d-1}\d r$
to both sides of the inequality, we get
\begin{align*}
 & \int_{0}^{R^{\prime}}\kappa(\sa v^{\prime}\sa^{q-2}v^{\prime}-\sa v_{\varepsilon}^{\prime}\sa^{q-2}v_{\varepsilon}^{\prime})(v^{\prime}-v_{\varepsilon}^{\prime})\xi^{q}r^{d-1}\d r\\
 & \ \leq\int_{0}^{R^{\prime}}\kappa q\sa v_{\varepsilon}^{\prime}\sa^{q-1}\sa v-v_{\varepsilon}\sa\sa\xi^{\prime}\sa\xi^{q-1}r^{d-1}\d r\\
 & \ \ \ \ \ +\int_{0}^{R^{\prime}}\left|v-v_{\varepsilon}\right|\xi^{q}\left|f_{\varepsilon}\right|r^{d-1}\d r\\
 & \ \ \ \ \ +\int_{0}^{R^{\prime}}\kappa\sa v^{\prime}\sa^{q-2}v^{\prime}(v^{\prime}-v_{\varepsilon}^{\prime})\xi^{q}r^{d-1}\d r\\
 & \ =:A_{1}+A_{2}+A_{3}.
\end{align*}
Since $v_{\varepsilon}^{\prime}$ is bounded in $L^{q}(r^{d-1},(0,R^{\prime}))$,
$v_{\varepsilon}\rightarrow v$ in $L^{q}(r^{d-1},(0,R^{\prime}))$
and $f\in L^{\infty}$, it follows from H{\"o}lder's inequality that $A_{1},A_{2}\rightarrow0$
as $\varepsilon\rightarrow0$. Moreover, since $v_{\varepsilon}^{\prime}\rightarrow v^{\prime}$
weakly in $L^{q}(r^{d-1},(0,R^{\prime}))$, also $A_{3}$ converges
to zero. We conclude that $v_{\varepsilon}^{\prime}\rightarrow v^{\prime}$
strongly in $L^{q}(r^{d-1},(0,R^{\prime\prime}))$ by applying H{\"o}lder's
inequality and the following inequality (see \cite[p95-96]{lindqvist_plaplace})
\[
(\left|a\right|^{q-2}a-\left|b\right|^{q-2}b)\left(a-b\right)\geq\begin{cases}
\left(q-1\right)\left|a-b\right|^{2}(1+\left|a\right|^{2}+\left|b\right|^{2})^{\frac{q-2}{2}}, & 1<q<2,\\
2^{2-q}\left|a-b\right|^{q}, & q\geq2.
\end{cases}
\]

Recall then that since $v_{\varepsilon}$ is a weak supersolution
to $-\kappa\Delta_{q}^{d}v_{\varepsilon}\geq f_{\varepsilon}$ in
$(0,R_{\varepsilon})$, any $\varphi\in C_{0}^{\infty}(-R^{\prime\prime},R^{\prime\prime})$
satisfies
\[
\int_{0}^{R^{\prime\prime}}\kappa\sa v_{\varepsilon}^{\prime}\sa^{q-2}v_{\varepsilon}^{\prime}\varphi^{\prime}r^{d-1}-\varphi f_{\varepsilon}r^{d-1}\d r\geq0.
\]
Since $v_{\varepsilon}^{\prime}\rightarrow v^{\prime}$ strongly in
$L^{q}(r^{d-1},(0,R^{\prime\prime}))$, we may let $\varepsilon\rightarrow0$
in the above inequality. Since $R^{\prime\prime}<R$ was arbitrary,
the proof is finished.
\end{proof}
\smallskip{}

\section{The case of integer $d$ }

We show that if $d$ is an integer, then weak supersolutions to (\ref{eq:radial eq})
coincide with radial weak supersolutions to $\smash{-\Delta_{q}u\geq f(\left|x\right|)}$,
where $\smash{\Delta_{q}}$ is the usual $q$-Laplacian in $d$-dimensions.
We begin by recalling the definition of weak supersolutions to the
latter equation. 
\begin{defn}
\label{def:weak q-Laplace}Let $d$ be an integer and let $\smash{B_{R}\subset\mathbb{R}^{d}}$
be a ball centered at the origin. A function $\smash{u\in W_{loc}^{1,q}(B_{R})}$
is a weak supersolution to $\smash{-\Delta_{q}u\geq f(\left|x\right|)}$
in $\smash{B_{R}}$ if 
\[
\int_{B_{R}}\left|Du\right|^{q-2}Du\cdot D\varphi-\varphi f(\left|x\right|)\d x\geq0
\]
for all non-negative $\smash{\varphi\in C_{0}^{\infty}(B_{R})}$.
\end{defn}
We will use the following lemma which states that the weighted Sobolev
space $\smash{W^{1,q}(r^{d-1},(0,R))}$ can be identified with the
space of radial Sobolev functions in $d$-dimensions. Similar results
hold also for higher-order Sobolev spaces, see \cite{figueiredoSantosMiyagaki}.
\begin{lem}
\label{lem:radial sobolev spaces}Let $d$ be an integer. Assume that
$\smash{u:B_{R}\rightarrow\mathbb{R}}$ is radial, i.e.\ $u(x)=v(\left|x\right|)$
for all $\smash{x\in B_{R}}$. Then $\smash{u\in W^{1,q}(B_{R})}$
if and only if $\smash{v\in W^{1,q}(r^{d-1},(0,R))}$. Moreover, we
have
\begin{equation}
Du(x)=\frac{x}{\left|x\right|}v^{\prime}(\left|x\right|)\quad\text{for a.e. }x\in B_{R}.\label{eq:radial sobolev spaces formula}
\end{equation}
\end{lem}
\begin{proof}
Suppose first that $\smash{v\in W^{1,q}(r^{d-1},(0,R))}$. By Lemma
\ref{thm:w1q density} there is a sequence $\smash{v_{n}\in C^{\infty}[0,R]}$
such that $\smash{v_{n}\rightarrow v}$ in $\smash{W^{1,q}(r^{d-1},(0,R))}$.
Setting $\smash{u_{n}(x):=v_{n}(\left|x\right|)}$ we have $\smash{u_{n}\in W^{1,q}(B_{R})}$
by Lipschitz continuity and
\begin{equation}
Du_{n}(x)=\frac{x}{\left|x\right|}v_{n}^{\prime}(\left|x\right|)\quad\text{for all }x\in B_{R}\setminus\left\{ 0\right\} .\label{eq:radial sobolev spaces 1}
\end{equation}
We obtain using the formula (9) in \cite[p280]{steinShakarchi05}
\begin{align*}
\int_{B_{R}}\left|u_{n}-u\right|^{q}\d x & =\int_{\partial B_{1}}\int_{0}^{R}\left|u_{n}(rz)-u(rz)\right|^{q}r^{d-1}\d r\d\sigma(z)\\
 & =\int_{\partial B_{1}}\int_{0}^{R}\left|v_{n}(\left|rz\right|)-v(\left|rz\right|)\right|^{q}r^{d-1}\d r\d\sigma(z)\\
 & =\sigma(\partial B_{1})\int_{0}^{R}\left|v_{n}(r)-v(r)\right|^{q}r^{d-1}\d r,
\end{align*}
where $\sigma$ is the spherical measure. Similarly, but now also
using (\ref{eq:radial sobolev spaces 1}), we compute
\begin{align*}
\int_{B_{R}}\left|Du_{n}-\frac{x}{\left|x\right|}v^{\prime}(\left|x\right|)\right|^{q}\d x= & \int_{B_{R}}\left|v_{n}^{\prime}(\left|x\right|)-v^{\prime}(\left|x\right|)\right|^{q}\d x\\
= & \sigma(\partial B_{1})\int_{0}^{R}\left|v_{n}^{\prime}(r)-v^{\prime}(r)\right|^{q}r^{d-1}\d r.
\end{align*}
Since $\smash{v_{n}\rightarrow v}$ in $\smash{W^{1,q}(r^{d-1},(0,R))}$,
it follows from the last two displays that $u\in W^{1,q}(B_{R})$
and that (\ref{eq:radial sobolev spaces formula}) holds.

Suppose then that $\smash{u\in W^{1,q}(B_{R})}$. Since $u$ is radial,
there exists a sequence of radial functions $\smash{u_{n}(x)=v_{n}(\left|x\right|)}$
such that $\smash{u_{n}\in C^{\infty}(B_{R})}$ and $\smash{u_{n}\rightarrow u}$
in $\smash{W^{1,q}(B_{R})}$. Now we have
\begin{align*}
\sigma(\partial B_{1})\int_{0}^{R}\left|v_{n}(r)-v(r)\right|^{q}r^{d-1}\d r & =\int_{\partial B_{1}}\int_{0}^{R}\left|v_{n}(\left|rz\right|)-v(\left|rz\right|)\right|^{q}r^{d-1}\d r\d\sigma(z)\\
 & =\int_{B_{R}}\left|u_{n}(x)-u(x)\right|^{q}\d x,
\end{align*}
which means that $\smash{v_{n}\rightarrow v}$ in $\smash{L^{q}(r^{d-1},(0,R))}$.
Observe then that for all $m,n\in\mathbb{N}$ we have
\begin{align*}
\sigma(\partial B_{1})\int_{0}^{R}\left|v_{n}^{\prime}(r)-v_{m}^{\prime}(r)\right|r^{d-1}\d r & =\int_{B_{R}}\left|v_{n}^{\prime}(\left|x\right|)-v_{m}^{\prime}(\left|x\right|)\right|^{q}\d x\\
 & =\int_{B_{R}}\left|\frac{x}{\left|x\right|}\cdot Du_{n}(x)-\frac{x}{\left|x\right|}\cdot Du_{m}(x)\right|^{q}\d x\\
 & \leq\int_{B_{R}}\left|Du_{n}(x)-Du_{m}(x)\right|^{q}\d x.
\end{align*}
In other words, $\smash{v_{n}}$ is Cauchy in $\smash{W^{1,q}(r^{d-1},(0,R))}$
and thus converges to some function. This function has to be $v$
since $\smash{v_{n}\rightarrow v}$ in $\smash{L^{q}(r^{d-1},(0,R))}$.
Hence we have established that $\smash{v\in W^{1,q}(r^{d-1},(0,R))}$.
The formula (\ref{eq:radial sobolev spaces formula}) now follows
from the first part of the proof.
\end{proof}
\begin{thm}
\label{thm:q laplace equivalence}Let $d$ be an integer. Then $v$
is a radial weak supersolution to (\ref{eq:radial eq}) in $(0,R)$
if and only if the function $u(x):=v(\left|x\right|)$ is a weak supersolution
to $-\Delta_{q}u=f(\left|x\right|)$ in $B_{R}\subset\mathbb{\mathbb{R}}^{d}$. 
\end{thm}
\begin{proof}
Suppose first that $v$ is a weak supersolution to (\ref{eq:radial eq})
in $(0,R)$. By Lemma \ref{lem:radial sobolev spaces} we have at
least $\smash{u\in W_{loc}^{1,q}(B_{R})}$. Let $\smash{\varphi\in C_{0}^{\infty}(B_{R})}$
be a non-negative test function. Then by \cite[p280]{steinShakarchi05}
and (\ref{eq:radial sobolev spaces formula}) we have
\begin{align*}
 & \int_{B_{R}}\left|Du\right|^{q-2}Du\cdot D\varphi-\varphi f(\left|x\right|)\d x\\
 & \ =\int_{\partial B_{1}}\int_{0}^{R}\left|Du(rz)\right|^{q-2}Du(rz)\cdot D\varphi(rz)r^{d-1}-\varphi(rz)f(\left|rz\right|)r^{d-1}\d r\d\sigma(z)\\
 & \ =\int_{\partial B_{1}}\int_{0}^{R}\left|v^{\prime}(r)\right|^{q-2}v^{\prime}(r)z\cdot D\varphi(rz)r^{d-1}-\varphi(rz)f(r)r^{d-1}\d r\d\sigma(z)\geq0,
\end{align*}
where the last inequality follows from the assumption that $v$ is
a weak supersolution to (\ref{eq:radial eq}) and that $\phi(r):=\varphi(rz)$,
$z\in\partial B_{1}$, is an admissible test function in Definition
\ref{def:weak sol}. Thus $u$ is a weak supersolution to $\smash{-\Delta_{q}u\geq f(\left|x\right|)}$
in $\smash{B_{R}}$. 

Suppose then that $u$ is a radial weak supersolution to $\smash{-\Delta_{q}u\geq f(\left|x\right|)}$
in $\smash{B_{R}}$. By Lemma \ref{lem:radial sobolev spaces} we
have $\smash{v\in W^{1,q}(r^{d-1},(0,R^{\prime}))}$ for all $\smash{R^{\prime}\in(0,R)}$.
Let $\smash{\phi\in C_{0}^{\infty}(-R,R)}$ be a non-negative test
function and set $\smash{\varphi(x):=\phi(\left|x\right|)}$. Then
$\varphi$ is a Lipschitz continuous function that is compactly supported
in $\smash{B_{R}}$ and therefore an admissible test function in Definition
\ref{def:weak q-Laplace}. Using formula (\ref{eq:radial sobolev spaces formula})
we obtain
\begin{align*}
0\geq & \int_{B_{R}}\left|Du\right|^{q-2}Du\cdot D\varphi-\varphi f(\left|x\right|)\d x\\
= & \int_{B_{R}}\left|\frac{x}{\left|x\right|}v^{\prime}(\left|x\right|)\right|^{q-2}v^{\prime}(\left|x\right|)\frac{x}{\left|x\right|}\cdot\frac{x}{\left|x\right|}\phi^{\prime}(\left|x\right|)-\phi(\left|x\right|)f(\left|x\right|)\d x\\
= & \int_{B_{R}}\left|v^{\prime}(\left|x\right|)\right|^{q-2}v^{\prime}(\left|x\right|)\phi^{\prime}(\left|x\right|)-\phi(\left|x\right|)f(\left|x\right|)\d x\\
= & \sigma(\partial B_{1})\int_{0}^{R}\left|v^{\prime}(r)\right|^{q-2}v^{\prime}(r)\phi^{\prime}(r)r^{d-1}-\phi(r)f(r)r^{d-1}\ dr,
\end{align*}
which means that $v$ is a weak supersolution to (\ref{eq:radial eq})
in $(0,R)$.
\end{proof}
Combining Theorems \ref{thm:weak is visc}, \ref{thm:visc is weak}
and \ref{thm:q laplace equivalence} we get the following corollary.
\begin{cor}
Let $d$ be an integer. Then $u(x):=v_{\ast}(\left|x\right|)$ is
a bounded viscosity supersolution to (\ref{eq:normplap f}) in $\smash{B_{R}\subset\mathbb{R}^{N}}$
if and only if $w(x):=v(\left|x\right|)$ is a bounded weak supersolution
to $\smash{-\Delta_{q}w=f(\left|x\right|)}$ in $B_{R}\subset\mathbb{R}^{d}$.
\end{cor}
\begin{rem}
Let us conclude this section with a brief remark on the special case
where $\eqref{eq:normplap f}$ is simply the homogeneous $p$-Laplace
equation ($q=p$ and $f\equiv0$). Recall that $p$-superharmonic
functions are defined as lower semicontinuous functions that satisfy
a comparison principle with respect to the solutions of the $p$-Laplace
equation \cite{lindqvist86}. In particular, the so called fundamental
solution
\[
V(x)=\begin{cases}
\left|x\right|^{\frac{p-N}{p-1}}, & p\not=N,\\
\log(\left|x\right|), & p=N,
\end{cases}
\]
is $p$-superharmonic. It is possible to show that if $u(x):=v(\left|x\right|)$
is a radial $p$-superharmonic function, then $v$ satisfies a comparison
principle with respect to weak solutions of (\ref{eq:radial eq}).
The converse is also true. If $v:[0,R)\rightarrow(-\infty,\infty]$,
$v\not\equiv\infty$, is a lower semicontinuous function that satisfies
a comparison principle with respect to weak solutions of (\ref{eq:radial eq}),
then $u$ is $p$-superharmonic. However, for expository reasons we
have decided to not discuss this further here.
\end{rem}

\appendix

\section{Some properties of the weighted Sobolev space}

In this section we collect some basic facts about the weighted Sobolev
space $W^{1,q}(r^{d-1},(0,R))$, where $d>1$. In particular, we have
the following theorem from \cite{kufner85} about the density of smooth
functions.
\begin{thm}
\label{thm:w1q density} The set
\[
C^{\infty}[0,R]:=\left\{ v_{|(0,R)}:v\in C^{\infty}(\mathbb{R})\right\} 
\]
is dense in $W^{1,q}(r^{d-1},(0,R))$.
\end{thm}
\begin{proof}
Let $v\in W^{1,q}(r^{d-1},(0,R))$. Take $\theta_{1}$,$\theta_{2}\in C^{\infty}(\mathbb{R})$
such that $0\leq\theta_{i}\leq1$, $\theta_{1}+\theta_{2}=1$ in $[0,R]$
and $\supp\theta_{1}\subset(-\infty,R^{\prime})$, $\supp\theta_{2}\subset(R^{\prime\prime},\infty)$
for some $0<R^{\prime\prime}<R^{\prime}<R$. Then we have
\[
v=\theta_{1}v+\theta_{2}v.
\]
Since $\theta_{2}v$ vanishes near zero, we have $\theta_{2}v\in W^{1,q}(0,R)$.
Hence by \cite[Theorem 8.2]{brezis2011} there exists a sequence of
functions in $C^{\infty}[0,R]$ that converges to $\theta_{2}v$ in
$W^{1,q}(0,R)$ and thus also in $W^{1,q}(r^{d-1},(0,R))$. Consequently
it remains to approximate the function
\[
w:=\theta_{1}v.
\]
For $\lambda>0$, we define the function $w_{\lambda}:(-\lambda,R)\rightarrow\mathbb{R}$
by setting
\[
w_{\lambda}(r):=w(r+\lambda).
\]
We show that $w_{\lambda}\rightarrow w$ in $W^{1,q}(r^{d-1},(0,R))$
as $\lambda\rightarrow0$. We start with the estimate
\begin{align}
 & \int_{0}^{R}\sa w_{\lambda}^{\prime}-w^{\prime}\sa^{q}r^{d-1}\d r\nonumber \\
 & \ =\int_{0}^{R}\sa w_{\lambda}^{\prime}r^{\frac{d-1}{q}}-w_{\lambda}^{\prime}\cdot(r+\lambda)^{\frac{d-1}{q}}+w_{\lambda}^{\prime}\cdot(r+\lambda)^{\frac{d-1}{q}}-w^{\prime}r^{\frac{d-1}{q}}\sa^{q}\d r\nonumber \\
 & \ \leq2^{q-1}\big(\int_{0}^{R}\sa w_{\lambda}^{\prime}\sa^{q}\sa r^{\frac{d-1}{q}}-(r+\lambda)^{\frac{d-1}{q}}\sa^{q}\d r+\int_{0}^{R}\sa w_{\lambda}^{\prime}\cdot(r+\lambda)^{\frac{d-1}{q}}-w^{\prime}r^{\frac{d-1}{q}}\sa^{q}\d r\big)\nonumber \\
 & =:2^{q-1}(I_{1}+I_{2}).\label{eq:density 1}
\end{align}
To see that $I_{1}\rightarrow0$ as $\lambda\rightarrow0$, fix $\varepsilon>0$.
Since $w^{\prime}\in L^{q}(r^{d-1},(0,R))$, we can take positive
$\delta=\delta(\varepsilon)<1$ such that 
\begin{equation}
\int_{0}^{2\delta}\left|w^{\prime}\right|^{q}r^{d-1}\d r<\varepsilon.\label{eq:density 0}
\end{equation}
Then for all $0<\lambda<\delta$ we have
\begin{align}
I_{1}= & \int_{0}^{R}\left|w^{\prime}(r+\lambda)\right|^{q}(r+\lambda)^{d-1}\bigg|1-\frac{r^{\frac{d-1}{q}}}{(r+\lambda)^{\frac{d-1}{q}}}\bigg|^{q}\d r\nonumber \\
\leq & \int_{0}^{\delta}\left|w^{\prime}(r+\lambda)\right|^{q}(r+\lambda)^{d-1}\d r+\int_{\delta}^{R}\left|w^{\prime}(r+\lambda)\right|^{q}(r+\lambda)^{d-1}\bigg|1-\frac{r^{\frac{d-1}{q}}}{(r+\lambda)^{\frac{d-1}{q}}}\bigg|^{q}\d r\nonumber \\
= & \int_{\lambda}^{\delta+\lambda}\left|w^{\prime}(r)\right|^{q}r^{d-1}\d r+\int_{\delta+\lambda}^{R+\lambda}\left|w^{\prime}(r)\right|^{q}r^{d-1}\bigg|1-\frac{(r-\lambda)^{\frac{d-1}{q}}}{r^{\frac{d-1}{q}}}\bigg|^{q}\d r\nonumber \\
\leq & \varepsilon+\int_{\delta}^{R+1}\left|w^{\prime}(r)\right|^{q}r^{d-1}\bigg|1-\frac{(r-\lambda)^{\frac{d-1}{q}}}{r^{\frac{d-1}{q}}}\bigg|^{q}\d r,\label{eq:density 4}
\end{align}
where in the last estimate we used (\ref{eq:density 0}). Since the
term
\[
\bigg|1-\frac{(r-\lambda)^{\frac{d-1}{q}}}{r^{\frac{d-1}{q}}}\bigg|^{q}
\]
is bounded by $1$ and converges to zero as $\lambda\rightarrow0$
for all $r>\delta$, it follows from Lebesgue's dominated convergence
theorem that for small enough $\lambda=\lambda(\varepsilon)<\delta$
we have
\begin{equation}
\int_{\delta}^{R+1}\left|w^{\prime}(r)\right|^{q}r^{d-1}\bigg|1-\frac{(r-\lambda)^{\frac{d-1}{q}}}{r^{\frac{d-1}{q}}}\bigg|^{q}\d r<\varepsilon.\label{eq:density 5}
\end{equation}
It follows from (\ref{eq:density 4}) and (\ref{eq:density 5}) that
$I_{1}\rightarrow0$ as $\lambda\rightarrow0$. Observe then that
\begin{equation}
I_{2}=\int_{0}^{R}\sa w^{\prime}(r+\lambda)(r+\lambda)^{\frac{d-1}{q}}-w^{\prime}(r)r^{\frac{d-1}{q}}\sa^{q}\d r=\int_{0}^{R}\sa g(r+\lambda)-g(r)\sa^{q}\d r,\label{eq:density 6}
\end{equation}
where $g(r)=w^{\prime}(r)r^{\frac{d-1}{q}}.$ Since $w^{\prime}\in L^{q}(r^{d-1},(0,R))$,
we have $g\in L^{q}(0,R)$. Thus $g$ is $q$-mean continuous by \cite[Theorem 3.3.3]{pickKufnerJohnFucik19}.
This means that the integral at the right-hand side of (\ref{eq:density 6})
converges to zero as $\lambda\rightarrow0$ and so also $I_{2}\rightarrow0$.
It now follows from (\ref{eq:density 1}) that $w_{\lambda}^{\prime}\rightarrow w^{\prime}$
in $L^{q}(r^{d-1},(0,R))$ and the convergence $w_{\lambda}\rightarrow w$
is seen in the same way. Consequently, for any $\varepsilon>0$ we
may take $\lambda_{\varepsilon}>0$ such that 
\begin{equation}
\left\Vert w_{\lambda_{\varepsilon}}-w\right\Vert _{W^{1,q}(r^{d-1},(0,R))}<\varepsilon.\label{eq:density 2}
\end{equation}
Observe now that $w_{\lambda_{\varepsilon}}\in W^{1,q}(-\mu,R)$ for
some $\mu\in(0,\lambda_{\varepsilon})$. Hence there is a function
$\psi\in C^{\infty}[-\mu,R]$ such that 
\begin{equation}
\left\Vert w_{\lambda_{\varepsilon}}-\psi\right\Vert _{W^{1,q}(-\mu,R)}<\varepsilon.\label{eq:density 3}
\end{equation}
Using (\ref{eq:density 2}) and (\ref{eq:density 3}) we obtain
\begin{align*}
\left\Vert w-\psi\right\Vert _{W^{1,q}(r^{d-1},(0,R))}\leq & \left\Vert w_{\lambda_{\varepsilon}}-\psi\right\Vert _{W^{1,q}(r^{d-1},(0,R))}+\left\Vert w_{\lambda_{\varepsilon}}-w\right\Vert _{W^{1,q}(r^{d-1},(0,R))}\\
\leq & \big(\int_{0}^{R}\sa w_{\lambda_{\varepsilon}}-\psi\sa^{q}r^{d-1}+\sa w_{\lambda_{\varepsilon}}^{\prime}-\psi^{\prime}\sa^{q}r^{d-1}\d r\big)^{1/q}+\varepsilon\\
\leq & R^{\frac{d-1}{q}}\big(\int_{0}^{R}\sa w_{\lambda_{\varepsilon}}-\psi\sa^{q}+\sa w_{\lambda_{\varepsilon}}^{\prime}-\psi^{\prime}\sa^{q}\d r\big)^{1/q}+\varepsilon\\
\leq & R^{\frac{d-1}{q}}\left\Vert w_{\lambda_{\varepsilon}}-\psi\right\Vert _{W^{1,q}(-\mu,R)}+\varepsilon\\
\leq & \varepsilon(R^{\frac{d-1}{q}}+1).
\end{align*}
Thus $w$ can be approximated by functions in $C^{\infty}[0,R]$ and
the proof is finished.
\end{proof}
\begin{lem}
\label{lem:sobolev is in H1q}The usual Sobolev space $W^{1,q}(0,R)$
is contained in $W^{1,q}(r^{d-1},(0,R))$. 
\end{lem}
\begin{proof}
If $v\in W^{1,q}(0,R)$, then $v$ has a distributional derivative
$v^{\prime}$. The claim then follows from the inclusion $L^{q}(0,R)\subset L^{q}(r^{d-1},(0,R))$
which holds since
\[
\int_{0}^{R}\left|v\right|^{q}r^{d-1}\d r\leq\int_{0}^{R}\left|v\right|^{q}R^{d-1}\d r.\qedhere
\]
\end{proof}
\begin{lem}
\label{lem:H1q conv}Let $\smash{v_{n}\in W^{1,q}(r^{d-1},(0,R))}$
be a sequence such that
\[
v_{n}\rightarrow v\text{ weakly in }L^{q}(r^{d-1},(0,R))
\]
and $\smash{v_{n}^{\prime}}$ is bounded in $\smash{L^{q}(r^{d-1},(0,R))}$.
Then $\smash{v\in W^{1,q}(r^{d-1},(0,R))}$ and
\[
v_{n}^{\prime}\rightarrow v^{\prime}\text{ weakly in }L^{q}(r^{d-1},(0,R))
\]
up to a subsequence.
\end{lem}
\begin{proof}
Since $v_{n}^{\prime}$ is bounded in $L^{q}(r^{d-1},(0,R))$, there
is $g\in L^{q}(r^{d-1},(0,R))$ such that $v_{n}^{\prime}\rightarrow g$
in $L^{q}(r^{d-1},(0,R))$ weakly up to a subsequence (see e.g.\ \cite[p126]{yosida80}).
Let $\varphi\in C_{0}^{\infty}(0,R)$. Then
\begin{align*}
\int_{0}^{R}g\varphi\d r=\int_{0}^{R}g\frac{\varphi}{r^{d-1}}r^{d-1}\d r & =\lim_{n\rightarrow\infty}\int_{0}^{R}v_{n}^{\prime}\frac{\varphi}{r^{d-1}}r^{d-1}\d r\\
 & =\lim_{n\rightarrow\infty}\int_{0}^{R}v_{n}^{\prime}\varphi\d r\\
 & =\lim_{n\rightarrow\infty}-\int_{0}^{R}v_{n}\varphi^{\prime}\d r\\
 & =-\int_{0}^{R}v\varphi^{\prime}\d r.
\end{align*}
Hence $g\in L^{q}(r^{d-1},(0,R))$ is the distributional derivative
of $v$, as desired.
\end{proof}
\begin{lem}
\label{lem:max is in H1q}If $v\in W^{1,q}(r^{d-1},(0,R))$, then
$v_{+}:=\max(v,0)\in W^{1,q}(r^{d-1},(0,R))$ with
\[
(v_{+})^{\prime}=\begin{cases}
v^{\prime} & \text{a.e.\ in }\left\{ r\in(0,R):v>0\right\} ,\\
0 & \text{a.e.\ in }\left\{ r\in(0,R):v\leq0\right\} .
\end{cases}
\]
\end{lem}
\begin{proof}
Let $\varphi\in C_{0}^{\infty}(0,R)$ with $\supp\varphi\subset I$,
where $I\Subset(0,R)$ is an interval. Since the restriction $v_{|I}$
is in the standard Sobolev space $W^{1,q}(I)$, we have also $(v_{|I})_{+}\in W^{1,q}(I)$
and
\[
((v_{|I})_{+})^{\prime}=\begin{cases}
v^{\prime} & \text{a.e.\ in }\left\{ r\in I:v>0\right\} ,\\
0 & \text{a.e.\ in }\left\{ r\in I:v\leq0\right\} .
\end{cases}
\]
Therefore 
\[
\int_{0}^{R}v_{+}\varphi^{\prime}\d r=\int_{I}(v_{|I})_{+}\varphi^{\prime}\d r=-\int_{I}((v_{|I})_{+})^{\prime}\varphi\d r=\int_{0}^{R}(v_{+})^{\prime}\varphi\d r.
\]
Since clearly $(v_{+})^{\prime}\in L^{q}(r^{d-1},(0,R))$, it follows
that $v\in W^{1,q}(r^{d-1},(0,R))$.
\end{proof}
\bibliographystyle{alpha}

\Addresses
\end{document}